\documentclass[11pt]{amsart}

\usepackage{amssymb, amsmath, amsthm, amsfonts, mathrsfs}
\usepackage{bbm}
\usepackage[dvipsnames]{xcolor}
\usepackage{comment}
\usepackage[normalem]{ulem}

\usepackage{hyperref}

\newtheorem{definition}{Definition}[section]
\newtheorem{theorem}[definition]{Theorem}
\newtheorem{corollary}[definition]{Corollary}
\newtheorem{proposition}[definition]{Proposition}
\newtheorem{lemma}[definition]{Lemma}
\newtheorem{remark}[definition]{Remark}
\newtheorem{assumption}{Assumption}

\newcommand{\R}{\mathbb{R}}
\newcommand{\E}{\mathbb{E}}
\newcommand{\N}{\mathbb{N}}

\newcommand{\di}{\mathrm{d}}
\newcommand{\eps}{\varepsilon}

\title[A {PDE} with negative {B}esov drift]{A PDE with  drift of  negative Besov index and linear growth solutions}

\author{Elena Issoglio and Francesco Russo}
\address[Elena Issoglio]{Dipartimento di Matematica `G.\ Peano', Universit\'a di Torino}
\email[Corresponding author]{elena.issoglio@unito.it}

\address[Francesco Russo]{Unit\'e de Math\'{e}matiques appliqu\'{e}es, ENSTA Paris, Institut Polytechnique de Paris}
\email{francesco.russo@ensta-paris.fr}

\date{\today}

\begin{document}

\begin{abstract}
   This paper investigates a class of PDEs with coefficients in negative Besov spaces and whose solutions have linear growth. We show existence and uniqueness of mild and weak solutions, which are equivalent in this setting, and several continuity results. To this aim, we introduce ad-hoc Besov-H\"older type spaces that allow for linear growth, and investigate the action of the heat semigroup on them.  We conclude the paper by introducing a special subclass of these spaces which has the useful property to be separable. 
\end{abstract}

\maketitle


\section{Introduction}

The objective of this paper is to study  existence, uniqueness and continuity results for solutions to a class of parabolic  PDEs with negative Besov drifts and unbounded solutions. 
In particular, the class of parabolic linear PDEs studied in this work  is of the form 
\begin{equation}\label{eq:parabolicPDE}
\left\{ 
\begin{array}{l}
\partial_t  v + \tfrac12 \Delta v +\nabla v \, b  =  \lambda v + g , \quad \text{on } [0,T]\times \R^d
\\
v(T) = v_T,
\end{array}\right. 
\end{equation}
where  $\lambda$ is a real parameter, and
$b$  and $g$ are continuous functions of time taking values in   a negative Besov space $\mathcal C^{(-\beta)+}$  with $0<\beta<\tfrac12 $, see definitions and details below. Here the product $\nabla v \, b := \nabla v \cdot b $ needs to be defined using pointwise products, because the term $b$ is a distribution. 

 Our main motivation to study \eqref{eq:parabolicPDE} comes from stochastic analysis. 
Indeed, PDEs of the form \eqref{eq:parabolicPDE} naturally arise in the context of stochastic differential equations,  particularly when setting and solving them as {\em martingale problems}. In the companion paper \cite{issoglio_russoMPb} we will extensively use all results on PDE \eqref{eq:parabolicPDE} found in the present paper. 


PDEs with distributional coefficients have been studied in the literature
before, see for example
\cite{cannizzaro,menozzi,issoglio13} to name a few.
Here we do not require the use of Gubinelli's paracontrolled distributions or Hairer's  regularity structures so that 
 the Besov index  of the space where the distributional coefficient $b$ lives cannot be lower than $-\frac12$. 
The main novelty  is
 that we allow terminal conditions, and hence the solutions,
 to have linear growth, which is not the case in the existing literature. 
 
 For this reason in Section \ref{sc:semigroup} we introduce a suitable class of functions spaces, denoted by $D\mathcal C^\alpha$, which contains all functions such that their derivative is an element of $\mathcal C^\alpha$. We also investigate the action of the semigroup on these spaces, in particular Schauder's estimates and Bernstein's inequalities  in the $D\mathcal C^\alpha $ spaces, see Lemmata \ref{lm:PtDC} and \ref{lm:PtDC2}.  
In Section \ref{sc:PDE} We introduce the notion of weak and mild solutions for this PDE and show that they are equivalent in Proposition \ref{pr:mildweak}. 
We then show existence and uniqueness of mild  solutions by fixed point arguments in Theorem \ref{thm:PDElambda0},  using properties of the heat semigroup. Furthermore we show  in Proposition \ref{pr:PDEproperties} some (uniform) bounds on the solution of a special case of the PDE, given by \eqref{eq:PDEu}. 
We also exhibiting several continuity results for the solutions with respect to the functions  $g, b, v_T$, both in the case when the solutions have linear growth and in the case when they are bounded. This is done in Lemmata \ref{lm:continuityv} and \ref{lm:continuityphi}.
In the last section we introduce and study a further class of spaces, which  are used in our companion paper \cite{issoglio_russoMPb} for applications in stochastic analysis,  together with all the results on the PDE. One of the important feature of these spaces is the fact that they are separable, which is not the case for the standard separable Besov-H\"older spaces.

\vspace{5pt}

The paper is organised as follows. In Section \ref{sc:prelim} we introduce the framework in which we work, define some tools like  the pointwise product and state some Assumptions. In Section   \ref{sc:semigroup} we define some new functions spaces that allow linear growth and derive useful properties of how the heat semigroup acts on them. The PDE \eqref{eq:parabolicPDE} is studied in Section  \ref{sc:PDE}.  In Section \ref{sc:separable} we introduce  and study a  class of Besov type spaces which is separable.

\section{Setting and preliminary results}\label{sc:prelim} 
\subsection{Function spaces}
We  use the notation $C^{0,1}:=C^{0,1}([0,T]\times \R^d)$ to indicate the space of functions with gradient in $x$  continuous in $(t,x)$.  By a slight abuse of notation we use the same notation 
$C^{0,1}$  for functions which are $\R^d$-valued. 
When $f:\R^d \to \R^d$ is differentiable, we denote by $\nabla f$ the matrix given by $(\nabla f)_{i,j} = \partial_i f_j$. When $f: \R^d \to \R$ we denote the Hessian matrix of $f$ by Hess$(f)$. Given any function $f$ defined on $[0,T]\times \R^d$ we often  denote  $f(t):= f(t, \cdot)$.

Let $\mathcal S=\mathcal S(\mathbb R^d )$ be the space of Schwartz functions on $\mathbb R^d$ and $\mathcal S'=\mathcal S'(\mathbb R^d )$ the space of Schwartz distributions. 
We denote by $\mathcal F$ and $\mathcal F^{-1}$
the Fourier transform on $\mathcal S$ and inverse Fourier transform respectively, which are extended to $\mathcal S'$ in the standard way.
 For $\gamma\in\mathbb R$ we denote by  $\mathcal C^\gamma = \mathcal C^\gamma(\mathbb R^d)$ the Besov space (or H\"older-Zygmund space) defined as
\begin{equation} \label{eq:Cgamma}
\mathcal C^\gamma := \left\{f \in \mathcal S' : \sup_{j\in\mathbb N} 2^{j\gamma} \|\mathcal F^{-1}(\varphi_j \mathcal F f)\|_\infty <\infty \right\},
\end{equation}
where $(\varphi_j)$ is some partition of unity.
$ \Vert \cdot \Vert_\infty $ denotes the usual sup-norm.
For more details see  for example 
\cite[Section 2.7]{bahouri}.  
Note that for $\gamma'<\gamma$ one has $\mathcal C^\gamma\subset\mathcal C^{\gamma'}$. 
If $\gamma \in \R^+ \setminus \mathbb N$ then the space coincides with the classical H\"older space, namely the space of bounded functions with bounded derivatives  up to order $ \left \lfloor{\gamma}\right \rfloor $ and such that the $ \left \lfloor{\gamma}\right \rfloor $th derivative is $ (\gamma - \left \lfloor{\gamma}\right \rfloor  )$-H\"older continuous. For example if $\gamma\in(0,1)$ the space $\mathcal C^\gamma$  can be equipped with   the classical $\gamma$-H\"older norm 
\begin{equation}\label{eq:gamma01}
\| f\|_\gamma := \|f\|_{\infty} + \sup_{x\neq y, |x-y|<1} \frac{|f(x)-f(y)|}{|x-y|^\gamma},
\end{equation} 
and if $\gamma \in (1,2)$ then  norm is given by 
\begin{equation}\label{eq:gamma12}
\|f\|_{\infty} + \|\nabla f\|_{\infty} +   \sup_{x\neq y, |x-y|<1} \frac{|\nabla f(x)-\nabla f(y)|}{|x-y|^\gamma}.
\end{equation}
We remark that it is equivalent in the previous formulations of the norms to take the supremum over the whole space rather than on $|x-y|<1 $. 
 Note that we use the same notation $\mathcal C^\gamma$ to indicate  $\R$-valued functions but also $\R^d$- or $\R^{d\times d}$-valued  functions. It will be clear from the context which space is needed.
 
 We denote by $C_T \mathcal C^\gamma$ the space of continuous functions on $[0,T]$ taking values in $\mathcal C^\gamma$, that is $C_T \mathcal C^\gamma:= C([0,T]; \mathcal C^\gamma)$.
For any given $\gamma\in \R$ we denote by $\mathcal C^{\gamma+}$ and $\mathcal C^{\gamma-}$  the spaces given by
\[
\mathcal C^{\gamma+}:= \cup_{\alpha >\gamma} \mathcal C^{\alpha} ,  \qquad  
\mathcal C^{\gamma-}:= \cap_{\alpha <\gamma} \mathcal C^{\alpha}.
\]
Note that $\mathcal C^{\gamma+}$ is an inductive space. 
We will also use the spaces $C_T C^{\gamma+}:=C([0,T]; \mathcal C^{\gamma+})$. We remark that $f\in C_T C^{\gamma+} $ if and only if there exists $\alpha>\gamma $ such that $f\in C_T \mathcal C^{\alpha}$, see \cite[Lemma B.2]{issoglio_russoMK}. 
Similarly, we use the space   $C_T \mathcal C^{\gamma-}:=C([0,T]; \mathcal C^{\gamma-})$; in particular we observe that  if $f\in C_T \mathcal C^{\gamma-} $ then for any $\alpha<\gamma $ we have $f\in C_T \mathcal C^{\alpha}$.   
Note that if $f$ is continuous and such that $\nabla f \in C_T \mathcal C^{0+}$ then $f\in C^{0,1}$.


Finally for a general Banach space $(B, \|\cdot \|_B)$ we introduce the family of $\rho$-equivalent norms on $C_T B$, denoted by $\|\cdot \|^{(\rho)}_{C_TB}$ and defined for all $\rho\geq 0$ by $\|f\|^{(\rho)}_{C_TB} = \sup_{t\in[0,T]} e^{-\rho (T-t)}  \|f(t)\|_B $. 
If $\rho = 0$ this is the standard norm in $C_T B$.

\subsection{The heat semigroup in $\mathcal S'$}

Let $(P_t)_t$ denote the semigroup generated by $\tfrac12\Delta$ on $\mathcal S$, in particular for all $\phi\in\mathcal S$ we define $(P_t \phi) (x):= \int_{\mathbb R^d} p_t(x-y) \phi(y)\mathrm dy$, where the kernel $p$ is the usual heat kernel $ p_t(x-y) =  \frac{1}{(2\pi t)^{d/2}} \exp\{ -\frac{|x-y|^2}{2t}\} $.  
It is easy to see that $P_t: \mathcal S \to \mathcal S$.
Moreover we can extend it to $\mathcal S'$  by dual pairing (and we denote it with the same notation by simplicity).  
One has $\langle P_t \psi, \phi\rangle = \langle \psi,  P_t \phi\rangle $ for each $\phi\in\mathcal S$ and  $\psi\in\mathcal S'$, using the fact that the kernel is symmetric.

Next we state and prove a joint continuity result for the heat semigroup acting on $\mathcal S'$. To this aim, we first recall some facts about the Schwartz space $\mathcal S'$, which is an inductive space. 
We recall that \cite[Section 7.3]{rudin} says that for any $\varphi\in \mathcal S, f \in \mathcal S'$ there exists a constant $C(f)$ and an integer $N\in \mathbb N$ such that 
\begin{equation}\label{eq:rudin}
|\langle   \varphi , f\rangle | \leq C(f) \sup_{y\in\R^d, |\alpha|\leq N} |D^\alpha \varphi (y)|\, (|y|^2+1)^N.
\end{equation}
From this it follows that the space $\mathcal S'(\R^d)$ can be expressed as the  space  $\mathcal S'(\R^d) = \cup_{N\in \N} E_N^*$ equipped with the inductive topology, where $E_N$ is the space of smooth functions $\varphi:\R^d \to \R$ such that 
\[
\|\varphi\|_N := \sup_{y\in \R^d, |\alpha|\leq N}  | D^\alpha \varphi(y) | \,  (|y|^2+1)^N < \infty. 
\]

\begin{lemma}\label{lm:rudin_t}
Let $f\in C_T \mathcal S'(\R^d)$.   Then there exists $N\in \mathbb N$ and a constant $C(f)$ independent of time such that 
\begin{equation*}\label{eq:rudin_t}
\sup_{t\in[0,T]} |\langle f(t), \varphi \rangle | \leq C(f)  \sup_{y\in\R^d, |\alpha|\leq N} |D^\alpha \varphi (y)|\, (|y|^2+1)^N.
\end{equation*}
for all $\varphi\in \mathcal S(\R^d)$.
In particular there exists $N$ such that $f\in C_T E^*_N$.
\end{lemma} 

\begin{proof}
Since $t\mapsto f(t) $ is continuous in $\mathcal S'(\R^d)$ then $(f(t))_{t\in[0,T]}$ is a compact in $\mathcal S'(\R^d)$, so there exists $N$ such that $f:[0,T]\to E^*_N$ and such that $(f(t))_{t\in[0,T]} $ is compact in $E^*_N$ by \cite[Remark B.1]{issoglio_russoMK}.  In particular, $(f(t))_{t\in[0,T]}$ is bounded in $E^*_N$, which implies that
\[
\sup_{t\in[0,T]} \| f(t)\|_{E_N \to \R} < C(f) <\infty,
\]
 and thus 
 \[
 \sup_{t\in[0,T]} |\langle f(t), \varphi \rangle |\leq C(f) \|\varphi\|_N = C(f) \sup_{y\in\R^d, |\alpha|\leq N} |D^\alpha \varphi (y)|\, (|y|^2+1)^N
 \]
for any $\varphi\in \mathcal S(\R^d)$.
\end{proof}

\begin{lemma}\label{lm:jointcontinuity}
  Let  $h\in C_T\mathcal S'$. Then the function $P_th(r)$ is jointly continuous in $(t,r)\in [0,T]^2$ with values in $\mathcal S'$.
\end{lemma}
\begin{proof}
By means of Fourier transform it is enough to prove that
$(r,t) \mapsto \mathcal F(P_th(r))$ is continuous
with values in $ \mathcal S'$.
We can write
\begin{equation}\label{eq:fourier}
\mathcal F(P_th(r)) (\xi ) = [\mathcal F(\exp(i t \cdot  )  ) \mathcal F h(r)] (\xi)= \exp(- \frac t 2 \xi^2   ) \mathcal F h(r) (\xi) .
\end{equation}
Expression \eqref{eq:fourier} has to be understood as an element of $\mathcal S'$.
When $t > 0$  the product of $\xi \mapsto \exp(- \frac t 2 \xi^2   ) \in \mathcal S$ and $\mathcal F h(r) \in \mathcal S' $ belongs to $\mathcal S'$.
In that case
\[
\mathcal F (P_t h (r))(\xi) = \langle( \mathcal F h(r))(\xi) , \exp(\frac{-t\xi^2}2) \rangle  \in \R,
\]
so that \eqref{eq:fourier} is a function.

We now prove that $(t,r) \mapsto \exp(\frac{-t\xi^2}2) \mathcal F h(r)(\xi)$ is continuous with values in $\mathcal S'$.
By Lemma \ref{lm:rudin_t} let $N$ be such that $\mathcal F h \in C_T E_N^*$ and let $(t_n,r_n) \to (t_0, r_0)$. Let $m\geq N$ to be chosen later.
We have (omitting the variable $\xi$  in $\mathcal F h(r)$ for brevity)
\begin{align*}
&\| \exp(- \frac {t_n} 2 \xi^2   ) \mathcal F h(r_n) - \exp(- \frac {t_0} 2 \xi^2   ) \mathcal F h(r_0)   \|_{E_m^*} \\
\leq & \| \exp(- \frac {t_n} 2 \xi^2   )[ \mathcal F h(r_n) - \mathcal F h(r_0)  ] \|_{E_m^*} \\
+&  \| [\exp(- \frac {t_n} 2 \xi^2   ) - \exp(- \frac {t_0} 2 \xi^2   )] \mathcal F h(r_0)   \|_{E_m^*} \\
=:& I_1(n) + I_2(n).
\end{align*}
We know that 
\[
I_1(n) = \sup_{\phi \in \mathcal S, \|\phi\|_{E_m}\leq 1} |\langle  \mathcal F h(r_n) - \mathcal F h(r_0) , \phi \exp(- \frac {t_n} 2 \xi^2   ) \rangle |.
\]
For $\phi \in \mathcal S $ we have
\[
 |\langle  \mathcal F h(r_n) - \mathcal F h(r_0) , \phi \exp(- \frac {t_n} 2 \xi^2   ) \rangle | \leq \|\mathcal F h(r_n) - \mathcal F h(r_0)\|_{E_N^*} \| \phi \exp(- \frac {t_n} 2 \xi^2   )\|_{E_N}
\]
and the first term goes to zero as $n\to \infty$ since $\mathcal F h \in C_T E^*_N$. We prove that 
\begin{equation}\label{eq:phiEm}
 \| \phi \exp(- \frac {t_n} 2 \xi^2   )\|_{E_N} \leq C_1 \|\phi \|_{E_{m_1}},
\end{equation}
for some $m_1\geq N$, where $C_1$ is a constant independent of $n$. Let $\alpha$ be a multi index such that $|\alpha|\leq N$. We have
\[
(1+ |\xi|^2)^{N} D^\alpha (\phi \exp(- \frac {t_n} 2 \xi^2   ))
\]
is a linear combination of 
\[
P(\xi; t_n) D^{\gamma} \phi (\xi)  \exp(- \frac {t_n} 2 \xi^2   )
\]
where $P(\xi;t_n)$ is a polynomial in $\xi $ with coefficient depending on $t_n$  that can be bounded from above by a polynomial in $\xi$ independent of $t_n$ (possibly depending on $T$)  and $|\gamma|\leq N$. It is clear that there exists an integer $m_1$ and a constant $C_1>0$ such that $P(\xi; t_n) \leq C_1(1+|\xi|^2)^{m_1}$. Thus \eqref{eq:phiEm} holds.

Concerning $I_2(n)$ we have 
\[
I_2(n) = \sup_{\phi \in \mathcal S, \|\phi\|_{E_m}\leq 1} |\langle  \mathcal F h(r_0),[ \exp(- \frac {t_n} 2 \xi^2   ) - \exp(- \frac {t_0} 2 \xi^2   )] \phi  \rangle |,
\]
so for  $\phi \in \mathcal S$ we have 
\begin{align*}
&|\langle  \mathcal F h(r_0),[ \exp(- \frac {t_n} 2 \xi^2   ) - \exp(- \frac {t_0} 2 \xi^2   )] \phi  \rangle |\\
 &\leq \| \mathcal F h(r_0) \|_{E^*_N} \frac{t_n-t_0}{2} \| \xi^2 \phi  \int_0^1 \exp(- \frac {t_n a + (1-a)t_0} 2 \xi^2  )\di a   \|_{E_N}.
\end{align*}
Since $ t_n-t_0 \to 0$ and $ \| \mathcal F h(r_0) \|_{E^*_N}$ is finite, it is enough to prove that  
  \begin{equation}\label{eq:phiEm2}
 \| \xi^2 \phi  \int_0^1 \exp(- \frac {t_n a + (1-a)t_0} 2 \xi^2  )\di a   \|_{E_N} \leq C_2 \|\phi \|_{E_{m_2}}
\end{equation}
for some $m_2$, where $C_2 $ is independent of $n$.
Let $\alpha$ be a multi index such that $|\alpha|\leq N$. 
Then
\[
(1+ |\xi|^2)^{N} D^\alpha \left (\xi^2 \phi(\xi) \int_0^1\exp(- \frac {t_n a + (1-a)t_0} 2 \xi^2  )\di a  \right)
\]
is a linear combination of terms of the type
\[
P(\xi; t_n) D^{\gamma} \phi (\xi) \int_0^1 \exp(- \frac {t_n a + (1-a)t_0} 2 \xi^2  )\di a 
\]
where  $P(\xi;t_n)$ is a polynomial in $\xi $ with coefficient depending on $t_n$  that can be bounded from above by a polynomial in $\xi$ independent of $t_n$ (possibly depending on $T$)  and $|\gamma|\leq N$. As above, there exists an integer $m_2$ and a constant $C_2>0$ such that $P(\xi; t_n) \leq C_2(1+|\xi|^2)^{m_2}$.  Thus\eqref{eq:phiEm2} holds.

Finally we conclude that $I_1(n) + I_2(n) \to 0$ as $n\to \infty$ by setting $m= m_1 \vee m_2$  and using the fact that the sequence of seminorms is monotone. 
\end{proof}

\begin{remark}\label{rm:nablaPt}
The semigroup $P_t$ and $\nabla$ commute  in $\mathcal S'$. 

Indeed let $h \in \mathcal S'$.  
We compute the (generalised) gradient of $P_th$, that is, for all $\phi\in\mathcal S$ we have
\begin{align*}
\langle \nabla P_t h , \phi \rangle :&=  - \langle  P_t h ,\text{div} \phi \rangle \\
&  = -\langle h , P_t \text{div} \phi \rangle \\
&  = -\langle h , \text{div}  P_t\phi \rangle \\
&  = \langle \nabla h ,   P_t\phi \rangle \\
&  = \langle P_t \nabla h ,   \phi \rangle.
\end{align*}
\end{remark}

\subsection{Estimates in $C^\gamma$ for the heat semigroup}
In this section, we are interested in the action of the semigroup on elements of Besov spaces $\mathcal C^\gamma$. 
These estimates are known as \emph{Schauder's estimates} (for a proof we refer to \cite[Lemma 2.5]{catellier_chouk}, see also \cite{gubinelli_imkeller_perkowski} for similar results). 
\begin{lemma}[Schauder's estimates]\label{lm:schauder}
Let $f\in \mathcal C^\gamma \subset \mathcal S'$ for some $\gamma \in \mathbb R$. Then for any $\theta\geq 0$ there exists a constant $c$ such that
\begin{equation}\label{eq:Pt}
\|P_t f\|_{\gamma + 2 \theta} \leq c t^{-\theta} \|f\|_\gamma
\end{equation}  
for all $t>0$.

Moreover 
for $f\in \mathcal C^{\gamma + 2\theta  }$ and   for any $\theta\in(0,1)$ we have 
\begin{equation}\label{eq:Pt-I}
\|P_t f-f\|_{\gamma} \leq c t^{\theta} \|f\|_{\gamma+2\theta }.
\end{equation}
\end{lemma}
Note that from \eqref{eq:Pt-I}, \eqref{eq:Pt} and the semigroup property,
 it readily follows that if $f\in \mathcal C^{\gamma + 2\theta }$ for some $0<\theta<1$, then for $t>s>0$ we have 
\begin{equation}\label{eq:PcontC}
\|P_tf-P_s f\|_{\gamma}\leq c (t-s)^{\theta} \|f\|_{\gamma+ 2\theta }.
\end{equation}
In other words, this means that if $f\in  \mathcal C^{\gamma+2\theta}$ then $P_\cdot f\in C_T \mathcal C^\gamma$ (and in fact it is $\theta$-H\"older continuous in time). 
We also recall that Bernstein's inequalities hold (see \cite[Lemma 2.1]{bahouri} and \cite[Appendix A.1]{gubinelli_imkeller_perkowski}), that is for $\gamma \in \mathbb R$ there exists a constant $c>0$ such that
\begin{equation}\label{eq:nabla}
\|\nabla g\|_{\gamma} \leq c  \|g\|_{\gamma+1},
\end{equation}  
for all $g\in \mathcal C^{1+\gamma}$. 
Using Schauder's and Bernstein's inequalities we can easily obtain a useful estimate on the gradient of the semigroup, as we see below. 
\begin{lemma}\label{lm:nablaP}
Let $\gamma\in \mathbb R$ and $\theta \in (0,1)$. If $g\in \mathcal C^\gamma$ then for all $t>0$ we have $\nabla (P_tg) \in \mathcal C^{\gamma +2\theta -1}$ and
\begin{equation}\label{eq:nablaP}
\|\nabla (P_tg)\|_{\gamma+2\theta-1} \leq c t^{-\theta}  \|g\|_{\gamma}.
\end{equation}  
\end{lemma}

\subsection{Further properties/tools}\label{ssc:pointwise}
The following is an important estimate which allows to define the pointwise product between certain distributions and functions, which is based on Bony's estimates. For details see \cite{bony} or \cite[Section 2.1]{gubinelli_imkeller_perkowski}. Let   $f \in \mathcal C^\alpha$ and $g\in\mathcal C^{-\beta}$ with $\alpha-\beta>0$ and $\alpha,\beta>0$. Then the {`pointwise product'} $ f \, g$ is well-defined as an element of $\mathcal C^{-\beta}$ and  there exists a constant $c>0$ such that 
\begin{equation}\label{eq:bony}
\| f \, g\|_{-\beta} \leq c \| f \|_\alpha \|g\|_{-\beta}.
\end{equation} 
\begin{remark}\label{rm:bonyt}
Using \eqref{eq:bony} it is not difficult to see that  if $f \in C_T \mathcal C^{\alpha}$ and $g \in C_T \mathcal C^{-\beta}$ then the product is also continuous with values in $ \mathcal C^{-\beta}$, and 
\begin{equation}\label{eq:bonyt}
\| f \, g\|_{C_T \mathcal C^{-\beta}} \leq c \| f \|_{C_T \mathcal C^{\alpha}} \|g\|_{ C_T \mathcal C^{-\beta}}.
\end{equation} 
\end{remark}

\section{The spaces $D \mathcal C^\gamma$ and the action of the semigroup}\label{sc:semigroup}

In this section we introduce some other  function spaces that will be central in the analysis of the  PDEs in this paper if we are to have solutions with linear growth. The idea is to have functions with the same regularity as the $\mathcal C^{\gamma}$-spaces locally, that allow linear growth at infinity. On these spaces
we will show how the heat semigroup acts in terms of regularity, both in the  time- and in the space-variable.

For $\gamma\in(0,1)$  we define space $D\mathcal C^\gamma $ as
\begin{align*}\label{eq:S}
  D\mathcal C^\gamma:= \{ & h:  \R^d \to \R \text{ differentiable function s.t. } \nabla h \in \mathcal C^\gamma\}.
\end{align*}
 Note that the following inclusion holds:
$$\mathcal C^{1+\alpha}  \subset D\mathcal C^\alpha.$$
On $ D\mathcal C^\alpha$  we can introduce a topology, induced by the norm
\begin{equation}\label{eq:normS}
\|h\|_{D\mathcal C^\gamma}:=  \left( |h(0)| + \|\nabla h\|_{\gamma} \right).
\end{equation}
If $h\in D\mathcal C^\alpha$ then there exists a constant (which is $h(0)$) and a function $\tilde h \in \mathcal C^\alpha$ (multidimensional) such that $h(x) = h(0) + x\cdot \tilde h$. Indeed, that function $\tilde h$ is given by $\int_0^1 \nabla h (ax) \mathrm da $. 
\begin{lemma}
 $(D\mathcal C^\alpha, \|\cdot \|_{D\mathcal C^\alpha})$ is a Banach space.
 \end{lemma}
\begin{proof}
  Let $(h_n)_n$ be a Cauchy sequence in $D\mathcal C^\alpha$. Then $h_n\in C^1$ and  since $\R$ and $\mathcal C^\alpha$ are complete,  we know that $h_n(0) \to c\in \mathbb R$ and  $ \nabla h_n \to g$ in $\mathcal C^\alpha$ hence uniformly. Now we  write
  $ h_n(x) = h_n(0)+ x \int_0^1 \nabla h_n (ax) \mathrm da$. We define  $ h(x) = c+ x \int_0^1 g (ax)\di a$, so that $\lim_{n\to\infty} h_n(x) = h(x)$. It is obvious that $c=h(0)$. Now we notice that  $\nabla h \in \mathcal S'$ so it is left to prove that $\nabla h =g$ in $\mathcal S'$ to conclude. For any test function $\phi \in \mathcal S$ we have $ \langle  \nabla h_n, \phi \rangle =\langle  h_n, -\text{div}( \phi) \rangle \to \langle  h, -\text{div}( \phi) \rangle $ as $n\to \infty$. On the other hand $ \langle  \nabla h_n, \phi \rangle \to  \langle  g, \phi \rangle$ hence we conclude $g= \nabla h$.
\end{proof}

 Next we study the mapping properties of the semigroup $P_t$ on $D\mathcal C^\alpha$  (and on the classical spaces $\mathcal C^{\alpha+1}$) for some fixed $\alpha \in(0,1)$.  
First we prove an inequality that is the analogous of Schauder's estimate \eqref{eq:Pt} with $\theta =0 $ on $D\mathcal C^\alpha$.
\begin{lemma}\label{lm:PtD}
If $h\in D\mathcal C^\alpha$, then
\begin{equation}\label{eq:PtD}
\sup_{s\in[0,T]} \|P_s h\|_{D\mathcal C^\alpha} \leq c \|h\|_{D\mathcal C^\alpha}.
\end{equation} 
\end{lemma}
\begin{proof} 
Using the definition of the  norm in $D \mathcal C^\alpha $ we have  
\begin{align*}
  \|P_s h \|_{D\mathcal C^\alpha} = |(P_s h)(0)| + \|\nabla P_s h \|_{\alpha} =: B_1 (s) + B_2(s).
\end{align*}
Using the kernel of the semigroup and writing $h(x)= h(0) + x \cdot  \int_0^1 \nabla h (a x) \di a  $ we get
\begin{align*}
B_1 (s) &= |\int_{\R^d} p_s( y) h(y) \di y  |\\
& \leq  |\int_{\R^d} p_s( y) h(0) \di y  | +  |\int_{\R^d} p_s( y) y\cdot \int_0^1 \nabla h (a y) \di a  \di y  |\\
& \leq | h(0)| 1 + \int_{\R^d} p_s( y) |y| \sup_x |\nabla h ( x) |  \di y  \\
&\leq | h(0)|+  c\|\nabla h\|_{\alpha} 
\leq  c\| h\|_{D \mathcal C^\alpha} .
\end{align*}
On the other hand, since $\nabla$ and $P_t$ commute by Remark \ref{rm:nablaPt}, we have 
\[
B_2(s) =  \|\nabla P_s h \|_{ \alpha}  	= \| P_s \nabla h \|_{ \alpha}  \leq c   \|\nabla h \|_{ \alpha} \leq c \| h \|_{D\mathcal C^\alpha},
\]
having used Schauder's estimate \eqref{eq:Pt}. This proves \eqref{eq:PtD}.
\end{proof} 
 
\begin{lemma}\label{lm:PtDC}
Let $\alpha\in(0,1)$. 
\begin{itemize}
\item[(i)]
  The semigroup $P_t$ maps  $\mathcal C^{\alpha+1}$ into itself. Moreover if $h\in \mathcal C^{1+\alpha+\nu}$ for some $\nu>0$ such that $\alpha+\nu \in (0,1) $, then $ P_\cdot h\in C_T \mathcal C^{1+\alpha}$.
\item[(ii)] The semigroup $P_t$ maps  $D\mathcal C^\alpha$ into itself. Moreover if $h\in D\mathcal C^{\alpha+\nu}$ for some $\nu>0$ such that $\alpha+\nu \in (0,1) $, then $ P_\cdot h\in C_T  D\mathcal C^\alpha$.
\end{itemize} 
\end{lemma} 
\begin{proof}

{\em Item (i)} This is an obvious consequence of Schauder's estimate (Lemma \ref{lm:schauder}) and equation \eqref{eq:PcontC}.

{\em Item (ii)} 
Let $h \in D\mathcal C^\alpha \subset \mathcal S'$. Let $t\geq 0$ be fixed.  
 By Remark \ref{rm:nablaPt}  $ \nabla P_t h  = P_t \nabla h$, so that $ \nabla P_t h  \in  \mathcal C^\alpha$ (and this automatically implies that  $P_t h$ is a differentiable function of $x$). 

Next we show that $t\mapsto P_t h$ is continuous with values in $D\mathcal C^\alpha$ if $h\in D\mathcal C^{\alpha+\nu}$. We need to show that for each $t\geq 0$ we have
\begin{align}\nonumber
 &\|P_{t+\eps} h - P_t h\|_{D\mathcal C^\alpha} \\ \label{eq:limDC}
&= |(P_{t+\eps} h)(0) - (P_t h)(0)| +\|\nabla P_{t+\eps} h - \nabla P_t h\|_{\alpha} \to  0 \text{ as }\varepsilon \to 0 .
\end{align}

Concerning first term in \eqref{eq:limDC}  we note that since $h\in D\mathcal C^{\alpha+\nu}$ then $\nabla h$ 
belongs to $ C^{\alpha+\nu},$
and $\| \nabla h \|_\infty \le \|\nabla h\|_{\mathcal C^{\alpha+\nu}}$. 
We observe that for any $t\geq0$ and $x\in \mathbb R^d$ we have $(P_th)(x) = \E [h(W^x_t)]$ where $(W^x_t)$ is a Brownian motion starting at $W_0 =x$.
Hence
\begin{align}\label{eq:b1}\notag
|(P_{t+\eps}h)(0)-(P_{t}h)(0)| &= |\E [h(W^0_{t+\eps}) - h(W^0_t)] |\\ \notag
 &\leq \E  [|h(W^0_{t+\eps}) - h(W^0_t)|] \\ \notag
 &\leq \|\nabla h\|_{\infty} \E  [|W^0_{t+\eps}- W^0_t|] \\ \notag
 &= \|\nabla h\|_{\alpha+\nu} \E  [|W^0_{\eps}|]\\
 &=\sqrt{\tfrac2\pi} \eps^{\tfrac12}\|\nabla h\|_{{\alpha+\nu}}.
\end{align}

The second term  in \eqref{eq:limDC}  can be bounded by 
\begin{equation}\label{eq:b2}
\|\nabla P_{t+\eps} h - \nabla P_t h\|_{\alpha}\leq c \eps^{\nu/2} \|\nabla h\|_{\alpha+\nu}
\end{equation}
 by using the fact that $\nabla$ and $P_t$ commute by Remark \ref{rm:nablaPt} together with \eqref{eq:PcontC} $\theta=\nu/2$. 

Putting \eqref{eq:b1} and \eqref{eq:b2} together we get
\[
\|P_{t+\eps} h - P_t h\|_{D\mathcal C^\alpha} \leq c \eps^{\frac{\nu\wedge 1}2} \|\nabla h\|_{{\alpha+\nu}}  \leq c \eps^{\frac{\nu\wedge 1}2} \|  h\|_{D\mathcal C^{\alpha+\nu}},
\]
which shows $P_\cdot h\in C_T D\mathcal C^{\alpha}$ as wanted.  
\end{proof}

\begin{lemma}\label{lm:PtDC2}
Let $\alpha \in(0,1)$.
\begin{itemize}
\item[(i)] Let $h\in C_T \mathcal C^{\alpha+1}$. Then  $\int_\cdot^T P_{s-\cdot}h(s) \di s\in C_T \mathcal C^{\alpha+1}$  and $\| \int_\cdot^T P_{s-\cdot} h(s) \di s\|_{C_T \mathcal C^{\alpha+1}}\leq c\| h\|_{C_T  \mathcal C^{\alpha+1}}$.
\item[(ii)]  Let $h\in C_T D\mathcal C^{\alpha}$. Then  
$\int_\cdot^T P_{s-\cdot}h(s) \di s\in C_T D\mathcal C^{\alpha}$  and $\| \int_\cdot^T P_{s-\cdot} h(s) \di s\|_{C_T D\mathcal C^{\alpha}}\leq c\| h\|_{C_T  D\mathcal C^{\alpha}}$.
\end{itemize}
\end{lemma}
\begin{proof}
We first  show that given $h\in C_T D\mathcal C^{\alpha}$ (resp.\ $h\in C_T \mathcal C^{\alpha+1}$), then  \\
$\int_0^{T-\cdot} P_{s}h(s+\cdot) \di s\in C_T D\mathcal C^{\alpha}$ (resp. $\int_0^{T-\cdot} P_{s}h(s+\cdot) \di s\in C_T \mathcal C^{\alpha+1}$), which is equivalent to the first part of the claim in (ii) (resp.\ in  (i)).
To this aim, let $t_n\to t_0$. We have
\begin{align}\label{eq:tnt0}
&\int_0^{T-t_n} P_s h(s+t_n) \di s - \int_0^{T-t_0} P_s h(s+t_0) \di s \nonumber  \\
&=\int_0^{T-t_0} P_s [h(s+t_n)- h(s+t_0)] \di s    + \int_{T-t_0}^{T-t_n} P_s h(s+t_n) \di s .   
\end{align}
We denote by $\delta (h; s)  $ the modulus of continuity of $h$ in $D\mathcal C^{\alpha} $ (resp.\  in $\mathcal C^{\alpha+1}$). Then the first integral in \eqref{eq:tnt0} is bounded in the $D\mathcal C^{\alpha}$-norm using \eqref{eq:PtD} (resp.\ in the $\mathcal C^{\alpha+1}$-norm using \eqref{eq:Pt} with $\theta=0$ and $\gamma = \alpha+1$) to get 
\begin{align*}
\|\int_0^{T-t_0} &P_s [h(s+t_n)- h(s+t_0)] \di s\|_{D\mathcal C^{\alpha}} \\
&\leq \int_0^{T-t_0} \| P_s [h(s+t_n)- h(s+t_0)] \|_{D\mathcal C^{\alpha}}\di s \\
&\leq c \int_0^{T-t_0} \|h(s+t_n)- h(s+t_0) \|_{D\mathcal C^{\alpha}}\di s \\
&= c \int_0^{T-t_0} \delta (h; t_n-t_0) \di s  \\
&=  c (T-t_0) \delta (h; t_n-t_0) ,
\end{align*}
respectively 
\[
\|\int_0^{T-t_0} P_s [h(s+t_n)- h(s+t_0)] \di s\|_{\mathcal C^{\alpha+1}} \leq c  (T-t_0) \delta (h; t_n-t_0), 
\]
which tends to 0 as $n\to \infty$. The second integral  in \eqref{eq:tnt0} is bounded again using \eqref{eq:PtD} (resp.\ using \eqref{eq:Pt} with $\theta=0$  and $\gamma = \alpha+1$) to get 
\begin{align*}
\|\int_{T-t_0}^{T-t_n} P_s h(s+t_n) \di s \|_{D\mathcal C^{\alpha}} 
&\leq \left| \int_{T-t_0}^{T-t_n} \|  P_s h(s+t_n) \|_{D\mathcal C^{\alpha}}\di s \right| \\
&\leq c \left |\int_{T-t_0}^{T-t_n} \|h(s+t_n) \|_{D\mathcal C^{\alpha}}\di s\right| \\
&= c \left| \int_{T-t_0}^{T-t_n} \|h\|_{C_T D\mathcal C^\alpha} \di s  \right|\\
&=  c |t_0-t_n| \|h\|_{C_T D\mathcal C^\alpha} , 
\end{align*}
respectively 
\[
\|\int_{T-t_0}^{T-t_n} P_s h(s+t_n) \di s \|_{\mathcal C^{\alpha+1}} \leq  c |t_0-t_n|  \|h\|_{C_T \mathcal C^{\alpha+1}}, 
\]
which tends to 0 as $n\to \infty$.

It is left to prove that $\| \int_\cdot^T P_{s-\cdot} h(s) \di s\|_{C_T D\mathcal C^{\alpha}}\leq c\| h\|_{C_T  D\mathcal C^{\alpha}}$ for point (ii) (resp.\ $\| \int_\cdot^T P_{s-\cdot} h(s) \di s\|_{C_T \mathcal C^{\alpha+1}}\leq c\| h\|_{C_T  \mathcal C^{\alpha+1}}$ for point (i)). Using again \eqref{eq:PtD} (resp.\ \eqref{eq:Pt} with $\theta=0$) we have 
\begin{align*}
\| \int_\cdot^T P_{s-\cdot} h(s) \di s\|_{C_T D\mathcal C^{\alpha}} &= \sup_{t\in[0,T]} \| \int_t^T P_{s-t} h(s) \di s\|_{D\mathcal C^{\alpha}}\\
 &\leq   \sup_{t\in[0,T]} \int_t^T \|P_{s-t} h(s) \|_{D\mathcal C^{\alpha}} \di s \\
 &\leq  c \sup_{t\in[0,T]} \int_t^T \| h(s) \|_{D\mathcal C^{\alpha}} \di s \\
  &\leq  c T \| h \|_{C_TD\mathcal C^{\alpha}}  ,
\end{align*}
respectively 
\[
\| \int_\cdot^T P_{s-\cdot} h(s) \di s\|_{C_T \mathcal C^{\alpha+1}}  \leq  c T \| h \|_{C_T\mathcal C^{\alpha+1}}, 
\]
which is the claim. 
 \end{proof}
In fact, it turns out that in the $\mathcal C^\alpha$ spaces a stronger continuity result will be needed, which is the following. 
\begin{lemma}\label{lm:PtS}
If  $h\in C_T \mathcal C^{(-\beta)+}$ then $\int_0^\cdot P_{\cdot-s} h(s) \di s \in C_T \mathcal C^{1+\alpha}$ with any $\alpha \in [\beta, 1-\beta) $.
\end{lemma}
\begin{proof}
  This is the adaptation of \cite[Lemma 3.2]{issoglio19} in the special case $h\in C_T \mathcal C^{-\beta}\subset L^\infty ([0,T]; \mathcal C^{-\beta} )$.
\end{proof}

Analogously as for the $\mathcal C^{\gamma+}$-spaces, for $\gamma>0$ we also introduce the spaces 
$$D \mathcal C^{\gamma+}:= \cup_{\alpha >\gamma} D\mathcal C^{\alpha} ,  \qquad  
D\mathcal C^{\gamma-}:= \cap_{\alpha <\gamma} D\mathcal C^{\alpha}.
$$
We will also use the spaces $C_T D\mathcal  C^{\gamma+}:=C([0,T]; D \mathcal C^{\gamma+})$. We remark that $f\in C_T D \mathcal  C^{\gamma+} $ if and only if there exists $\alpha>\gamma $ such that $f\in C_T D \mathcal  C^{\alpha}$, see \cite[Remark B.1]{issoglio_russoMK}. 
Similarly, we use the space   $C_T D C^{\gamma-}:=C([0,T]; D \mathcal C^{\gamma-})$; we observe in particular that if $f\in C_T D \mathcal C^{\gamma-} $ then for any $\alpha<\gamma $ we have $f\in C_T D\mathcal C^{\alpha}$.

\section{Main results}\label{sc:PDE}

In this section we prove existence, uniqueness, continuity  properties and various bounds for solutions to a class of parabolic  PDEs with unbounded terminal condition. This means that said solutions too are unbounded, indeed they live in the space $C_T D\mathcal C^{\beta+} $. We also consider a special case of this class where terminal conditions are bounded, hence also the solutions are bounded, i.e.\ they live in $C_T \mathcal C^{(1+\beta)+} $. 

\subsection{Assumptions}

We introduce here various assumptions concerning dis\-tri\-bu\-tion\--valued functions ($b$ respectively $g$) needed below in the paper.

\begin{assumption}\label{ass:param-b} 
Let $0<\beta<1/2$ and  $b\in C_T \mathcal C^{(-\beta)+}(\R^d)$. In particular $b\in C_T \mathcal C^{-\beta}(\R^d)$. Notice that $b$ is a column vector.
\end{assumption}


Next we introduce two assumptions concerning  $g $ and $ v_T$.

\begin{assumption}\label{ass:PDEv}
We suppose that $g\in C_T \mathcal C^{(-\beta)+}$  and $v_T \in  D\mathcal C^{(1-\beta)-} $. 
\end{assumption} 

\begin{assumption}\label{ass:PDE}
We suppose that  $g\in C_T \mathcal C^{(-\beta)+}$  and $v_T \in  \mathcal C^{(2-\beta)-} $.
\end{assumption}

The main difference between  Assumption \ref{ass:PDE}  and Assumption \ref{ass:PDEv} is that in the latter we allow the terminal condition to be unbounded, in particular we can choose $v_T = \text{id}$, while in the former the identity function is excluded.

\subsection{A class of PDEs with drifts in Besov spaces}
Let $b$ fulfill Assumption \ref{ass:param-b} for the rest of Section \ref{sc:PDE}.
 Let $v_T\in\mathcal S'$ and $t\mapsto g(t,\cdot)$ be continuous in $\mathcal S'$.  We consider here PDEs of the form
\begin{equation}\label{eq:PDE}
\left\{ 
\begin{array}{l}
\partial_t  v + \tfrac12 \Delta v +\nabla v \, b  =  \lambda v + g
\\
v(T) = v_T.
\end{array}\right. 
\end{equation}
We consider weak and mild solutions,  both defined in the space $ C_T  D\mathcal C^{\beta}  $,
 as detailed below. To shorten notation, we define 
 \[
G(v) := \lambda v + g. 
 \]

\begin{definition}\label{def:weak} 
Let  $v\in C_T  D \mathcal C^{\beta}  $. 
We say that $v$ is a weak solution of \eqref{eq:PDE} if for all $\varphi \in \mathcal S(\R^d)$  we have that $v$ satisfies
\begin{align}\label{eq:weak}
&\int_{\R^d} \varphi(x) v_T(x) \di x -\int_{\R^d} \varphi(x)v(t,x) \di x + \int_t^T \int_{\R^d} \frac12\Delta \varphi(x) v(s,x) \di x \, \di s \\ \nonumber
&+ \int_t^T \int_{\R^d} \varphi(x) \left(\nabla v(s,x) b(s,x) \right) \di x \di s   \nonumber 
= \int_t^T \int_{\R^d} \varphi(x) {G(v)(s,x)} \di x \di s, \nonumber
\end{align}
for all $t\in[0,T]$.
\end{definition}
Notice that the notation $\int_{\R^d} \varphi(x) \left(\nabla v(s,x) b(s,x)\right) \di x$ is only formal because $\nabla v(s,\cdot) b(s,\cdot)$ is a distribution. In practice when we write the integral we mean the dual pairing with $\varphi$, namely  $\langle \varphi,\nabla v(s) b(s)\rangle  $, where the pairing in $\mathcal S, \mathcal S'$ is well-defined as an element in $ \mathcal C^{(-\beta)+}$    via the pointwise product \eqref{eq:bony}.

\begin{definition}\label{def:mild}
Let  $v\in C_T  D \mathcal C^{\beta}  $. 
We say that $v$ is a mild solution of \eqref{eq:PDE} if $v$ satisfies
\begin{equation}\label{eq:mild}
v(t) = P_{T-t} v_T  + \int_t^T P_{s-t} \left(\nabla v(s) b(s) \right) \di s - \int_t^T P_{s-t}  (G(v)(s))\di s,
\end{equation}
for all $t\in[0,T]$.
\end{definition}
Note that  for each $s\in[0,T]$ the product $\nabla v(s) \, b(s)$ appearing in \eqref{eq:weak} and \eqref{eq:mild} is well-defined as an element of $\mathcal C^{(-\beta)+}$ using the pointwise product \eqref{eq:bony}, thanks to Assumption \ref{ass:param-b}.  Indeed since $v \in C_T D \mathcal C^{\beta} $ and $b \in C_T \mathcal C^{-(\beta)+}$  we can always choose $\eps>0$ such that $b \in C_T \mathcal C^{-\beta+ \eps}$ so that $-\beta+\eps+\beta = \eps>0$ and \eqref{eq:bonyt} holds. 
Moreover both integrals are  well-defined as Bochner integrals with values  in $\mathcal S' $ because  $(s,r)\mapsto P_sh(r)$ is jointly continuous with values in $\mathcal S'$ (where $h$ is either $\nabla v\, b$ or $G(v)$, and the continuity follows from Lemma \ref{lm:jointcontinuity}).

For future use, it is convenient to properly define the singular operator  $\mathcal L$, formally given by 
$\mathcal L f = \partial_t f + \tfrac12\Delta f+ \nabla f \, b$.

\begin{definition}\label{def:L}
Let $b$ satisfy Assumption \ref{ass:param-b}. The operator $\mathcal L$ is defined as  
\begin{equation*}
\begin{array}{lcll}
\mathcal L  :  &\mathcal D_{\mathcal L}^0 &\to &\{\mathcal S'\text{-valued continuous functions}\}\\
& f & \mapsto & \mathcal L f:=   \dot f + \frac12\Delta f + \nabla f \, b,
\end{array}
\end{equation*}
where 
$$\mathcal D_{\mathcal L}^0 : = C_T  D\mathcal C^{\beta} \cap C^1([0,T]; \mathcal S').$$
Here $f: [0,T]\times  \R^d \to \R$ and the function $\dot f:[0,T]\to \mathcal S'$ is the time-derivative of $f$. Note also that $\nabla f \, b $ is well-defined   and continuous using  \eqref{eq:bonyt}  and Assumption \ref{ass:param-b}.
 The Laplacian $\Delta$ is intended in the weak sense.
\end{definition}

\begin{remark}\label{rm:AC}
We observe that if $v\in C_T D \mathcal C^{\beta} $  is a weak solution, then it is automatically differentiable in time with continuous derivative in $\mathcal S'$, hence $v\in \mathcal D_{\mathcal L}^0 .$ The same is true for  $v\in C_T \mathcal C^{1+\beta} $ by the inclusion of the spaces. 
\end{remark}

Using the operator $\mathcal L$ defined in Definition \ref{def:L} and Remark \ref{rm:AC}, we see that PDE \eqref{eq:PDE} rewrites as
\[
\left\{ 
\begin{array}{l}
\mathcal L  v   =  \lambda v + g
\\
v(T) = v_T.
\end{array}\right.
\]

\begin{proposition}\label{pr:mildweak}
Weak and mild solutions of \eqref{eq:PDE} are equivalent in  $C_T D \mathcal C^{\beta} $.
\end{proposition}
\begin{proof}
{\em (i) mild implies weak.} 
Let $v\in C_T D\mathcal C^{\beta} $ be a mild solution. For any $\varphi\in\mathcal S$ we have
\begin{align}\label{eq:mw}
\nonumber
\int_t^T \langle v(s),\frac12\Delta \varphi \rangle \di s 
= &\int_t^T \langle P_{T-s} v_T , \frac12\Delta \varphi \rangle \di s \\
&+ \int_t^T  \int_s^T \langle P_{r-s} \nabla v (r) b(r) , \frac12 \Delta \varphi \rangle \di r  \di s \\
& - \int_t^T  \int_s^T \langle P_{r-s} G(v)(r) ,\frac12 \Delta \varphi \rangle \di r  \di s . \nonumber
\end{align}
The first term on the RHS of \eqref{eq:mw} gives 
\begin{align*}
\int_t^T \langle P_{T-s} v_T, \frac12 \Delta \varphi \rangle \di s &= \int_t^T \langle  \frac12\Delta P_{T-s}  v_T , \varphi \rangle \di s\\
&= \int_0^{T-t} \langle  \frac{\di}{\di s} P_{s}  v_T, \varphi \rangle \di s\\
&= \langle P_{T-t} v_T, \varphi \rangle   -\langle v_T, \varphi \rangle.
\end{align*}
The second and third terms on the RHS of \eqref{eq:mw} give
\begin{align*}
\int_t^T  \int_s^T\langle P_{r-s} & [\nabla v(r) b(r) - G(v)(r)], \frac12\Delta \varphi \rangle \di r \di s  \\
= & \int_t^T  \int_0^{r-t} \langle P_{s} [\nabla v(r) b(r) -G(v)(r) ], \frac12 \Delta \varphi \rangle \di s \di r\\
 =& \int_t^T  \int_0^{r-t} \langle \frac{\di }{\di s} P_{s} [\nabla v(r) b(r) - G(v)(r)]  , \varphi \rangle \di s \di r\\
= &\int_t^{T} \langle   P_{r-t}  [\nabla v(r) b(r) - G(v)(r)], \varphi \rangle \di r \\
&-  \int_t^{T} \langle [\nabla v(r) b(r) -   G(v)(r)], \varphi \rangle \di r.
\end{align*}
 Putting these into \eqref{eq:mw}   we get
 \begin{align*}
\int_t^T \langle v(s), \frac12 \Delta \varphi \rangle \di s 
=& \langle P_{T-t} v_T, \varphi \rangle   -\langle v_T, \varphi \rangle +\int_t^{T} \langle   P_{r-t} [\nabla v(r) b(r) - G(v)(r)], \varphi \rangle \di r\\
&-  \int_t^{T} \langle [\nabla v(r) b(r) -  G(v)(r)], \varphi \rangle \di r\\
=&\langle v(t) , \varphi\rangle -\langle v_T, \varphi \rangle +  \int_t^{T} \langle  [\nabla v(r) b(r) -G(v)(r)], \varphi \rangle \di r
\end{align*}
 which shows that $v$ is also a weak solution.  

{\em (ii) weak implies mild.} We proceed as follows. Given a weak solution  $v\in  C_T D\mathcal C^{\beta} $  that satisfies \eqref{eq:weak} we define 
\begin{equation} \label{eq:uDef}
u(t) := P_{T-t} v_T + \int_t^T P_{s-t} \left(\nabla v(s) b(s) \right) \di s - \int_t^T P_{s-t}  G(v)(s)  \di s.
\end{equation}
We see that $u$ is a mild solution of the heat equation with extra source terms involving $v$, more specifically of  
\[
\partial_t u +\frac12\Delta u = G(v) - \nabla v \,b; \qquad u(T) = v_T.
\]
By using  (i) with $\lambda =0$ and $g= G(v) - \nabla v \,b$ we have that $u$ is  also a weak solution of the above PDE.  Now we take the difference $\bar v = v-u $ and see that $\bar v$ fulfills 
\[
\bar v (t, \cdot ) =  -\int_t^T \frac12 \Delta \bar v(s, \cdot)\di s; \qquad \bar v(T) = 0,
 \]
hence $\bar v$ is  a weak solution of the heat equation with zero terminal condition so we have $\bar v =0$, which implies that $u=v$ and so $u $ is a mild solution by  \eqref{eq:uDef}.
\end{proof}

\subsection{Linear growth solutions}
In this subsection we consider equation  \eqref{eq:PDE}  and pick a terminal condition  $v_T$ fulfilling   Assumption \ref{ass:PDEv}. 
We will show below  that solutions of \eqref{eq:PDE} exist  in the space $C_T D  \mathcal C^{(1-\beta)-}$ and are unique in the space $C_T D  \mathcal C^{\beta+}$. If furthermore the terminal condition is bounded (Assumption \ref{ass:PDE}) then the solution will also be bounded. 

Given  $\rho\geq 0$, we introduce an equivalent norm in $C_T D\mathcal C^\alpha$, respectively $C_T \mathcal C^{\alpha+1}$, defined as
\begin{equation}\label{eq:rho01}
\|f\|_{ C_T D\mathcal C^\alpha}^{(\rho)}:= \sup_{t\in[0,T]} e^{-\rho (T-t)} \left( |f(t,0)| + \|\nabla f(t)\|_{\alpha} \right),
\end{equation}
respectively 
\begin{equation}\label{eq:rho12}
\|f\|_{ C_T \mathcal C^{\alpha+1}}^{(\rho)}:= \sup_{t\in[0,T]} e^{-\rho (T-t)} \left( \sup_x |f(t,x)| + \|\nabla f(t)\|_{\alpha} \right).
\end{equation}
Notice that those norms are equivalent to those
defined in
\eqref{eq:normS} (resp. \eqref{eq:gamma12}). 
With these norms the pointwise products estimates corresponding to those from Remark \ref{rm:bonyt} will become, for $\alpha>\beta$,
\begin{equation}\label{eq:bonytrho}
\|fg\|_{C_T \mathcal C^{-\beta}}^{(\rho)} \leq c \|f\|^{(\rho)}_{C_T \mathcal C^{\alpha}} \|g\|_{C_T \mathcal C^{-\beta}}.
\end{equation}

We start with a preliminary result.
\begin{lemma}\label{lm:Ptnablawh}
 Let  $\ell\in C_T \mathcal C^{-\beta}$ and $\rho\geq 1$. Then for every $t\in [0,T]$ and for every $\alpha \in [\beta, 1-\beta)$
we have \begin{equation}\label{eq:Ptnablawh}
 \|\int_t^T  P_{s-t} \ell (s) \di s \|^{(\rho)}_{C_T \mathcal C^{\alpha+1}} \leq c \|\ell\|_{C_T \mathcal C^{-\beta} }^{(\rho)} \rho^{\frac{\alpha+\beta-1}2},
 \end{equation}
and in particular,
\[
\|\int_t^T  P_{s-t}\ell (s)  \di s \|^{(\rho)}_{C_T D\mathcal C^\alpha} \leq c \|\ell \|_{C_T \mathcal C^{-\beta} }^{(\rho)} \rho^{\frac{\alpha+\beta-1}2},
 \]
 where $c$ depends on $\alpha$ and $\beta$.
\end{lemma}
\begin{proof}
We recall that for $f \in C_T \mathcal C^{\alpha+1}$ then $f \in C_T D \mathcal C^{\alpha}$ and $\|f\|_{ C_T D \mathcal C^{\alpha}} \leq \|f\|_{ C_T \mathcal C^{\alpha+1}} $ by   \eqref{eq:rho01}  and  \eqref{eq:rho12}.  For this reason, we will only prove \eqref{eq:Ptnablawh}.
We bound  each term in the $\rho $-equivalent norm \eqref{eq:rho12} in
  $\mathcal C^{\alpha+1}$ separately.
 Let us denote by $f(t,x):= \int_t^T \left( P_{s-t} \ell (s) \right) (x) \di s $.
The  sup term  in \eqref{eq:rho12}
 gives 
\begin{align*}
 \sup_x |f(t,x)|
& = \sup_{x}  \left | \int_t^T \left( P_{s-t}  \ell(s) \right) (x) \di s\right |\\
& = \left \| \int_t^T  P_{s-t}  \ell(s)  \di s\right\|_\infty\\
& \leq  \left  \| \int_t^T  P_{s-t}  \ell(s) \di s \right\|_\alpha\\
& \leq  \int_t^T   \left  \|P_{s-t}  \ell(s) \right\|_\alpha  \di s\\
& \leq c  \int_t^T   (s-t)^{-\frac{\alpha+\beta}2}  \|  \ell(s) \|_{ -\beta} \di s,
\end{align*}
having used \eqref{eq:Pt} from Lemma \ref{lm:schauder}.
Now multiplying by $e^{-\rho (T-t)}$ and taking the supremum over $t$, using \eqref{eq:rho12} 
 we get
\begin{align}\label{eq:rhots} \nonumber 
 \sup_{t\in [0,T]} &e^{-\rho (T-t)} \sup_x  |f(t,x)| \\ \nonumber 
  &\leq c    \sup_{t\in [0,T]}  \int_t^T e^{-\rho (s-t)}  (s-t)^{-\frac{\alpha+\beta}{2}}  e^{-\rho (T-s)}\| \ell(s) \|_{-\beta}  \di s\\ 
 &  \leq c \|\ell \|^{(\rho)}_{C_T \mathcal C^{-\beta}}  \sup_{t\in [0,T]} \int_t^T e^{-\rho (s-t)}      (s-t)^{-\frac{\alpha+\beta}{2}} \di s.
 \end{align}
The latter  integral can be bounded noting that  $\theta := \frac{\alpha+\beta}{2}<1$ by choice of $\alpha$, thus 
 \begin{align*} 
 \int_t^T e^{-\rho (s-t)}     (s-t)^{-\theta} \di s 
 & \leq \int_0^{\infty} e^{-s\rho} s^{-\theta} \di s \\
 & \leq \int_0^{\infty} e^{-x} x^{-\theta} \rho^{-1+\theta} \di x \\
  & = \Gamma(-\theta+1)  \rho^{-1+\theta},
 \end{align*} 
 where \begin{equation}\label{eq:defGamma}
\Gamma(\eta):=\int_0^{\infty} e^{-x} x^{\eta-1}  \di x 
 \end{equation}
 denotes the Gamma function. 
Thus \eqref{eq:rhots} gives 
\begin{equation}\label{eq:A2} 
 \sup_{t\in [0,T]} e^{-\rho (T-t)}  \sup_x|f(t,x) |   \leq c  \|\ell\|^{(\rho)}_{C_T \mathcal C^{-\beta}} \rho^{\frac{\alpha + \beta -2}{2}},
\end{equation}
where 
$c$ depends on $\alpha$ and $\beta$.

The term with the $\alpha$-norm 
of $\nabla f$  in \eqref{eq:rho12} is bounded 
 with similar computations as above but using \eqref{eq:nablaP} in place of \eqref{eq:Pt}
to get
\begin{align*}
\sup_{t\in [0,T]}& e^{-\rho (T-t)}\|\nabla f(t) \|_\alpha \\
& \leq  c\sup_{t\in [0,T]} e^{-\rho (T-t)} \int_t^T  (s-t)^{-\frac{\alpha+\beta+1}{2}}  \|  \ell(s) \|_{-\beta}  \di s .
\end{align*}
Proceeding as between \eqref{eq:rhots} and \eqref{eq:A2}  
 and using the fact that $\frac{\alpha+\beta+1}{2}<1$,
 we get 
\begin{align} \label{eq:AA1}
\sup_{t\in [0,T]} e^{-\rho (T-t)} \|\nabla f(t)\|_\alpha  
  &   \leq c \|\ell\|^{(\rho)}_{ C_T \mathcal C^{-\beta
}} \rho^{ \frac{\alpha + \beta -1}{2}}.
 \end{align} 
Combining   \eqref{eq:A2} and  \eqref{eq:AA1}, and using the fact that $\rho^{ \frac{\alpha + \beta -2}{2}} \leq \rho^{ \frac{\alpha + \beta -1}{2}}$ since $\rho\geq1$, we conclude.
 \end{proof}

\begin{theorem}\label{thm:PDElambda0}
 Let $b$ satisfy Assumption \ref{ass:param-b}.
 \begin{itemize}
 \item[(i)]  Let $v_T$ and $g$ satisfy Assumption \ref{ass:PDEv}. 
 Then there exists a  mild solution $v$ to \eqref{eq:PDE} in $C_T D  \mathcal C^{(1-\beta)-}$ which is  unique in $C_T D  \mathcal C^{\beta}$.
\item[(ii)] Let $v_T$ and $g$ satisfy Assumption \ref{ass:PDE} (in particular
  $v_T$ is bounded). Then the unique mild solution $v$ of PDE \eqref{eq:PDE} is also bounded, more precisely $v\in C_T \mathcal C^{(2- \beta)-}$.
 \end{itemize}
\end{theorem}
\begin{remark}\label{rm:PDElambda0}
 One could  relax Assumption  \ref{ass:PDEv} (resp.\ Assumption \ref{ass:PDE}) for $v_T$ and only ask that $v_T \in  D\mathcal C^{\beta+} $ (resp. $v_T \in  \mathcal C^{(1+\beta)+} $). In this case the unique solution would no longer  belong to $C_T D  \mathcal C^{(1-\beta)-}$ (resp.\ $C_T \mathcal C^{(2-\beta)-}$) but only to $C_T D  \mathcal C^{\beta+}$ (resp.\ $C_T  \mathcal C^{(1+\beta)+}$).
\end{remark} 
\begin{proof}[Proof of Theorem \ref{thm:PDElambda0}]
We start with an arbitrary $\alpha\in (\beta, 1-\beta)$. The case $\alpha=\beta$ will be explained at the end of the proof. 
Let $\mathcal T$ denote the solution operator, namely for $v\in C([0,T]; D\mathcal C^\alpha)$ we define $\mathcal T v$ as 
\begin{equation}\label{eq:Tau}
\mathcal T v(t):= P_{T-t} v_T+ \int_t^T P_{s-t} \left(\nabla v(s) b(s) \right) \di s - \int_t^T P_{s-t} (\lambda v(s) + g(s)) \di s.
\end{equation}
We prove both items of the theorem in two steps, first showing stability and then the contraction property. Notice that Assumption \ref{ass:PDE} implies Assumption \ref{ass:PDEv}.

\textbf{Step 1  - stability.} 
We suppose Assumption \ref{ass:PDEv} (resp.\ Assumption \ref{ass:PDE}).
We show that  $\mathcal T : C_T D\mathcal C^\alpha \to C_T D\mathcal C^\alpha$ (resp.\  $\mathcal T : C_T \mathcal C^{\alpha+1} \to C_T \mathcal C^{\alpha+1}$). \\
The term $P_{T-t} v_T \in D\mathcal C^\alpha$ (resp.\ $P_{T-t} v_T \in \mathcal C^{\alpha+1}$)  is continuous in $t$ by Lemma \ref{lm:PtDC}, item (ii)
(resp.\ item (i)) since $v_T\in D\mathcal C^{\alpha+\nu}$ (resp.\  $v_T\in \mathcal C^{1+\alpha+\nu}$) for all $\nu>0$ such that $\alpha+\nu<1-\beta$ by Assumption \ref{ass:PDEv} (resp.\ Assumption \ref{ass:PDE}).\\ 
Since $v\in C_T D\mathcal C^\alpha$  (resp.\ $v\in C_T \mathcal C^{\alpha+1}$) and $b \in C_T \mathcal C^{(-\beta)+}$, then by Remark \ref{rm:bonyt} $\nabla v b\in C_T\mathcal C^{(-\beta)+}$. 
Moreover $ g\in C_T \mathcal C^{(-\beta)+}$
 by assumption. 
Thus we can apply Lemma \ref{lm:PtS} to deduce that $\int_\cdot^T P_{s-\cdot} \left(\nabla v(s) b(s) \right) \di s + \int_\cdot^T P_{s-\cdot} g(s) \di s \in   C_T \mathcal C^{\alpha+1} \subset C_T D\mathcal C^\alpha  $.

Finally  by Lemma \ref{lm:PtDC2} item (ii)  (resp. item (i)), $t\mapsto \int_t^T P_{s-t} \lambda v(s) \di s$ is continuous with values in $ D\mathcal C^\alpha$ (resp. $\mathcal C^{\alpha+1}$).

\textbf{Step 2 - contraction.} Next we show that $\mathcal T$ is a contraction in $  C_T D\mathcal C^\alpha$ (resp.\ $ C_T \mathcal C^{\alpha+1}$).

To this aim it is convenient to use the equivalent norm in $C_T D\mathcal C^\alpha$ (resp.\ $ C_T \mathcal C^{\alpha+1}$) introduced in \eqref{eq:rho01} (resp.\ \eqref{eq:rho12}).
Let $v_1, v_2 \in  C_T D\mathcal C^\alpha $ (resp. $v_1, v_2 \in  C_T \mathcal C^{\alpha+1} $). Then 
\begin{align} \nonumber
\mathcal Tv_1(t) - \mathcal Tv_2(t) =& \int_t^T  P_{s-t} \left((\nabla v_1(s)-\nabla v_2(s) ) b(s) \right) \di s \\ \nonumber
&+ \lambda \int_t^T  P_{s-t} ( v_1(s)- v_2(s) )  \di s \\
\label{eq:BB}
&=: B_1(t) + B_2(t).
\end{align}
We consider $B_1$ first. By Lemma \ref{lm:Ptnablawh} with $\ell= \nabla (v_1-v_2)b$ and using \eqref{eq:bonytrho} we get
\begin{align}\label{eq:Te1}
\nonumber
\|B_1 \|^{(\rho)}_{C_T D\mathcal C^\alpha} 
&=
\|\int_t^T  P_{s-t} \left(\nabla (v_1-v_2)(s)  b(s) \right) \di s \|^{(\rho)}_{C_T D\mathcal C^\alpha}\\ \nonumber 
&\leq c \|\nabla (v_1-v_2)b\|_{C_T \mathcal C^{-\beta} }^{(\rho)} \rho^{\frac{\alpha+\beta-1}2}  \\ \nonumber 
&\leq  c \|b\|_{C_T \mathcal C^{-\beta} } \|\nabla (v_1-v_2)\|^{(\rho)}_{ C_T \mathcal C^\alpha} \rho^{\frac{\alpha+\beta-1}2},
\\
&\leq c \|b\|_{C_T \mathcal C^{-\beta} } \|v_1-v_2\|^{(\rho)}_{ C_T D\mathcal C^\alpha} \rho^{\frac{\alpha+\beta-1}2},
\end{align}
 respectively 
 \begin{equation}\label{eq:Te1ii}
 \|B_1 \|^{(\rho)}_{C_T \mathcal C^{\alpha+1}}  \leq c \|b\|_{C_T \mathcal C^{-\beta} } \|v_1-v_2\|^{(\rho)}_{ C_T \mathcal C^{\alpha+1}} \rho^{\frac{\alpha+\beta-1}2}.
 \end{equation}

We now bound $B_2$ in \eqref{eq:BB}. We use Lemma \ref{lm:PtD} (resp.\ Schauder's estimate \eqref{eq:Pt} with $\theta=0$) to get  
\begin{align}
\nonumber
  \|B_2\|_{C_T D \mathcal C^\alpha}^{(\rho)} & =
  \sup_{t\in[0,T]} e^{-\rho(T-t)}\| \lambda \int_t^T  P_{s-t} ( v_1(s)- v_2(s) )
    \di s\|_{D \mathcal C^\alpha}\\ \nonumber 
& \leq  \lambda  \sup_{t\in[0,T]} \int_t^T e^{-\rho(T-t)}  \| P_{s-t} ( v_1(s)- v_2(s) ) \|_{D \mathcal C^\alpha}  \di s\\ \nonumber 
& \leq  \lambda \sup_{t\in[0,T]} \int_t^T  e^{-\rho(s-t)}  c e^{-\rho(T-s)}  \|  v_1(s)- v_2(s)  \|_{D \mathcal C^\alpha}  \di s\\ \nonumber
& \leq c   \lambda  \sup_{t\in[0,T]}  \int_t^T  e^{-\rho(s-t)}   \|  v_1- v_2  \|^{(\rho)}_{C_T D \mathcal C^\alpha}  \di s\\ \nonumber
&\leq c  \lambda\|  v_1- v_2  \|^{(\rho)}_{C_T D \mathcal C^\alpha}  \rho^{-1}\\
& \leq c \lambda\|  v_1- v_2  \|^{(\rho)}_{C_T D \mathcal C^\alpha}  \rho^{\frac{\alpha+\beta-1}2}, \label{eq:Te3}
\end{align}
respectively 
\begin{equation}\label{eq:Te3ii}
\|B_2\|_{ \mathcal C^{\alpha+1}}^{(\rho)} \leq c\lambda \|  v_1- v_2  \|^{(\rho)}_{C_T \mathcal C^{\alpha+1}}  \rho^{\frac{\alpha+\beta-1}2}.
\end{equation}
  Combining \eqref{eq:Te1}  and \eqref{eq:Te3} (resp.\ \eqref{eq:Te1ii} and \eqref{eq:Te3ii}) and plugging them in \eqref{eq:BB} we get
 \begin{equation} \label{eq:stimarho1}
 \| \mathcal Tv_1 -  \mathcal Tv_2 \|_{ C_T D\mathcal C^\alpha}^{(\rho)} \leq c(\lambda+ \|b\|_{C_T \mathcal C^{-\beta} } ) \rho^{\frac{\alpha+\beta-1}2} \| v_1 - v_2 \|_{ C_TD\mathcal C^\alpha}^{(\rho)},
 \end{equation}
 respectively 
 \begin{equation} \label{eq:stimarho2}
  \|\mathcal Tv_1 -  \mathcal Tv_2 \|_{ C_T \mathcal C^{\alpha+1}}^{(\rho)} \leq c(\lambda+ \|b\|_{C_T \mathcal C^{-\beta} } ) \rho^{\frac{\alpha+\beta-1}2} \| v_1 - v_2 \|_{ C_T \mathcal C^{\alpha+1}}^{(\rho)}. 
 \end{equation}
 Now choosing  $\rho$ large enough
 so that (recalling that $\frac{\alpha+\beta-1}2<0$)
\begin{equation}\label{eq:rhob}
 c(\lambda+ \|b\|_{C_T \mathcal C^{-\beta} } ) \rho^{\frac{\alpha+\beta-1}2}  \leq \frac12,
\end{equation}
we get  
 \begin{equation}
\label{eq:fix}
 \| \mathcal Tv_1 -  \mathcal Tv_2 \|_{ C_T D\mathcal C^\alpha}^{(\rho)}  \leq \frac12  \| v_1 - v_2 \|_{ C_TD\mathcal C^\alpha}^{(\rho)},
 \end{equation}
 respectively 
 \begin{equation}
\label{eq:fixi}
 \| \mathcal Tv_1 -  \mathcal Tv_2 \|_{ C_T \mathcal {C}^{\alpha+1}}^{(\rho)}  \leq \frac12  \| v_1 - v_2 \|_{ C_T \mathcal C^{\alpha+1}}^{(\rho)},
 \end{equation}
 for all $v_1, v_2 \in   C_T D\mathcal C^\alpha $ (resp.\ $v_1, v_2 \in C_T \mathcal C^{\alpha+1}$). By Banach fixed point theorem we conclude that there exists a unique fixed point $v\in C_T D\mathcal C^\alpha$ (resp.\ in  $C_T \mathcal C^{\alpha+1}$) of $\mathcal T$, which is the  unique mild solution $v\in C_T D\mathcal C^\alpha$ to \eqref{eq:PDE}.
 
Since this is true for all $\alpha \in (\beta, 1-\beta)$, then under Assumption \ref{ass:PDE} existence holds in the smaller space $C_T D  \mathcal C^{(1-\beta)-}$. 
At this point we observe that we can choose $\alpha=\beta $ in all computations above,  but one must replace  $\|b\|_{C_T \mathcal C^{-\beta}}$ with  $\|b\|_{C_T \mathcal C^{-\beta+\eps}}$ 
for some small $\eps$ such that $2\beta-\eps+1>0$, and the powers $ \rho^{\frac{\alpha+\beta-1}2}$   must be  replaced by  $ \rho^{\frac{2\beta-\eps-1}2}$. In conclusion \eqref{eq:fix} and \eqref{eq:fixi} still hold for $\alpha=\beta$, hence   and uniqueness holds in the larger space $C_T D  \mathcal C^{\beta}$, 
   which proves item (i). 
  Moreover when Assumption  \ref{ass:PDEv} holds then  the unique solution belongs to $ C_T  \mathcal C^{(2-\beta)-}$.
\end{proof}

\begin{lemma}\label{lm:tauv}
Let $b$ satisfy Assumption \ref{ass:param-b}, $v_T$ and $g$ satisfy Assumption \ref{ass:PDE} and let $\lambda >0$.
Let $\alpha \in (\beta, 1-\beta)$  such that $v_T \in  \mathcal C^{\alpha+1}$.  Let $v$ be the unique solution of \eqref{eq:PDE} given in Theorem  \ref{thm:PDElambda0} item (ii) and Remark \ref{rm:PDElambda0}. Then there exists an  increasing function $R_\lambda$ such that 
\[
\| v  \|_{ C_T \mathcal C^{\alpha+1}}  \leq R_\lambda( \| b\|_{ C_T \mathcal C^{-\beta}}  ) ( \| v_T  \|_{ \mathcal C^{\alpha+1}} + \|g\|_{C_T \mathcal C^{-\beta}}).
\]
\end{lemma} 
\begin{proof}
For the map $\mathcal T$  defined in \eqref{eq:Tau} 
we have
\begin{align*}
 \| \mathcal Tv  \|_{ C_T \mathcal C^{\alpha+1}}^{(\rho)} &\leq  \| \mathcal Tv  - \mathcal T0\|_{ C_T \mathcal C^{\alpha+1}}^{(\rho)} +\| \mathcal T0\|_{ C_T \mathcal C^{\alpha+1}}^{(\rho)}.
\end{align*}
Using  \eqref{eq:stimarho2} with  $v_1 =v$ and $v_2 =0$   we get
\[
\| \mathcal Tv  - \mathcal T0\|_{ C_T \mathcal C^{\alpha+1}}^{(\rho)} \leq c(\lambda + \| b\|_{ C_T \mathcal C^{-\beta}} ) \rho^{-\theta} \| v  \|_{ C_T \mathcal C^{\alpha+1}}^{(\rho)},
\]
where $\theta =\frac{1-\alpha-\beta}2 >0$. 
On the other hand 
\[
\mathcal T0 = P_{T-t} v_T+ \int_t^T P_{s-t} g(s) \di s,
\]
so using Lemma \ref{lm:Ptnablawh}
\begin{align*}
\| \mathcal T0\|_{ C_T \mathcal C^{\alpha+1}}^{(\rho)} &\leq \|  P_{T-t}v_T\|_{ C_T \mathcal C^{\alpha+1}}^{(\rho)} +  \|\int_\cdot^T P_{s-\cdot} g(s) \di s  \|_{ C_T \mathcal C^{\alpha+1}}^{(\rho)} \\
&\leq  \| P_{T-t} v_T  \|_{  C_T \mathcal C^{\alpha+1}}  + c\| g \|_{ C_T \mathcal C^{-\beta}}^{(\rho)} \rho^{-\theta} \\
&\leq  \|  v_T  \|_{  \mathcal C^{\alpha+1}} + c\| g \|_{ C_T \mathcal C^{-\beta}}^{(\rho)} \rho^{-\theta}.
\end{align*}
Combining the estimates above we have 
\[
  \| \mathcal Tv  \|_{ C_T \mathcal C^{\alpha+1}}^{(\rho)} \leq c(\lambda+ \| b\|_{ C_T \mathcal C^{-\beta}} ) \rho^{-\theta} \| v  \|_{ C_T \mathcal C^{\alpha+1}}^{(\rho)} + \|  v_T  \|_{  \mathcal C^{\alpha+1}}  + c\| g \|_{ C_T \mathcal C^{-\beta}}^{(\rho)} \rho^{-\theta}.
\]
Choosing   $\rho =  [2c(\lambda+ \| b\|_{ C_T \mathcal C^{-\beta}})]^{1/\theta}$  so that 
$c(\lambda+ \| b\|_{ C_T \mathcal C^{-\beta}} ) \rho^{-\theta} = \frac12$
we get
\[
\| \mathcal Tv  \|_{ C_T \mathcal C^{\alpha+1}}^{(\rho)} \leq \frac12 \| v  \|_{ C_T \mathcal C^{\alpha+1}}^{(\rho)} +\| v_T  \|_{ \mathcal C^{\alpha+1}}  + \frac c {2c(\lambda+ \| b\|_{ C_T \mathcal C^{-\beta}})} \|g\|^{(\rho)}_{ C_T \mathcal C^{-\beta}} .
\]
Since $v$ is a solution then $\mathcal Tv =v $ and  we get 
\begin{align}\label{eq:vTrho} \nonumber
  \| v  \|_{ C_T \mathcal C^{\alpha+1}}^{(\rho)} &\leq 2 \| v_T  \|_{ \mathcal C^{\alpha+1}} + 2\frac 1 {2 (\lambda+ \|b\|_{C_T \mathcal C^{-\beta}})}  \|g\|^{(\rho)}_{C_T \mathcal C^{-\beta}}\\ 
  &\leq 2 \| v_T  \|_{ \mathcal C^{\alpha+1}}   + \frac1\lambda \|g\|^{(\rho)}_{C_T \mathcal C^{-\beta}}.
\end{align}
Using   $\| v \|_{C_T \mathcal C^{\alpha+1}} = \sup_{t\in[0,T]}e^{\rho(T-t)} e^{-\rho(T-t)} \|v(t)\|_{\mathcal C^{\alpha+1}} \leq  e^{\rho T}\| v \|^{(\rho)}_{C_T \mathcal C^{\alpha+1}}$, the bound \eqref{eq:vTrho} and  $ \| g  \|_{ C_T \mathcal C^{\alpha+1}}^{(\rho)} \leq  \|g \|_{C_T \mathcal C^{\alpha+1}}$ 
we get
\begin{align*}
 \| v  \|_{ C_T \mathcal C^{\alpha+1}} & \leq e^{\rho T}  \| v  \|_{ C_T \mathcal C^{\alpha+1}}^{(\rho)}\\
 & \leq   2 e^{\rho T} \| v_T  \|_{ \mathcal C^{\alpha+1}} +e^{\rho T}\frac1\lambda \|g\|_{C_T \mathcal C^{-\beta}}^{(\rho)}  \\
 &\leq  2 e^{\rho T} \max\{1, \frac1\lambda\} ( \| v_T  \|_{ \mathcal C^{\alpha+1}}+ \|g\|_{C_T \mathcal C^{-\beta}}).
\end{align*}
Recall that we chose  $\rho =  [2c(\lambda+ \| b\|_{ C_T \mathcal C^{-\beta}})]^{1/\theta}$  and since $\theta>0 $ the result follows with $R_\lambda(x): =2 \exp\{[2c(\lambda+ x)]^{1/\theta} T\} \max\{1, \frac1\lambda\} $.
\end{proof}

A special case of interest of PDE \eqref{eq:PDE} is the following. 
Let $\text{id}_i(x) = x_i,$ which clearly belongs to  $\mathcal D^0_{\mathcal L}$, see Definition \ref{def:L}.
Thus  $\mathcal L \, \text{id}_i$ is well-defined and gives $\mathcal L \, \text{id}_i = b_i$.  An immediate consequence of  Theorem \ref{thm:PDElambda0} point (i) with $\lambda = 0, v_T =  x_i, g= b_i$ is the following corollary, taking into account that $\text{id}_i \in C_T D \mathcal C^{\beta} $.
\begin{corollary}\label{cor:Lid=b}
The function  $\text{id}_i$ is the  solution of $\mathcal L v = b_i$; $v(T) = \text{id}_i$ (unique in $ C_T D \mathcal C^{\beta} $). 
\end{corollary}

\subsection{Properties of the solution: bounds and continuity}

Another particular
case of interest of PDE \eqref{eq:PDE} is given when $g$ is chosen to be  the $i$th component of the drift $b$ 
 and the terminal condition is zero.  We denote  by $u_i$ the solution in this
 case, that is 
\begin{equation}\label{eq:PDEu}
\left\{
\begin{array}{l}
\partial_t u_i + \tfrac12 \Delta  u_i+ \nabla  u_ib = \lambda u_i - b_i\\
u_i(T) = 0.
\end{array}
\right.
\end{equation}

\begin{remark}\label{rm:PDEexistence}
  Since PDE \eqref{eq:PDEu} is a special case of \eqref{eq:PDE} where $g=-b_i$ and $v_T =0$, by Theorem \ref{thm:PDElambda0} the solution $u_i$ exists in $C_T\mathcal C^{(2-\beta)-}$  and is unique in $C_T D \mathcal C^{\beta}$
  (indeed Assumption \ref{ass:PDE} is automatically satisfied for $v_T$ and $g$ if $b$ satisfies Assumption \ref{ass:param-b}). 
\end{remark}

\begin{remark}
Let $b \in  C_T \mathcal C^{0+}$.  Then the unique solution $u$ to \eqref{eq:PDEu} coincides with the classical solution in  $C^{1, 2+\nu}$ (see \cite[Theorem 5.1.9]{lunardi95}, see also \cite[Theorem A.3]{issoglio_russoMPb}).

Indeed, if $b \in  C_T \mathcal C^{0+}$ then $ b \in C^{0,\nu}([0,T] \times \mathbb R^d)$ for some $\nu >0$ by \cite[Remark A.2]{issoglio_russoMPb}, so by \cite[Theorem 5.1.9]{lunardi95}   there exists a (unique) solution $\bar u$ in  $  C^{1, 2+\nu}$ to PDE \eqref{eq:PDEu}. Moreover $b \in C_T \mathcal C^{0+} \subset C_T \mathcal C^{(-\beta)+}$ hence $u$ is the unique solution of \eqref{eq:PDEu} in $C_T \mathcal C^{(1+\beta)+}$. We moreover have the inclusion $  C^{1, 2+\nu} \subset C_T \mathcal C^{(1+\beta)+}$, thus  $\bar u = u \in  C^{1, 2+\nu}$.
\end{remark}

\begin{proposition}\label{pr:PDEproperties}
Let $b$ satisfy Assumption \ref{ass:param-b}, in particular $b\in C_T \mathcal C^{-\beta+\eps}$ for some  $\eps>0$ such that  $\theta:=\frac{1+2\beta-\eps}{2}<1$. Let $u_i$, $i=1, \ldots,d$ be the unique solution of \eqref{eq:PDEu} as given in Remark \ref{rm:PDEexistence}. Then the following holds.
\begin{itemize}
\item[(i)] The solution $u_i$ is bounded in $(t,x)$, that is, there exists a constant $c$ such that
\[
\sup_{(t,x)\in[0,T]\times \mathbb R^d} |u_i(t,x)| \leq c.
\]
\item[(ii)]  There is a constant $C(\beta, \eps) $ such that choosing   $\lambda$ with
\begin{equation}\label{eq:lambda}
{\lambda}^{1-\theta} \geq C(\beta, \eps) \|b\|_{C_T\mathcal C^{-\beta+\eps} },
\end{equation}
then we have 
\[
\sup_{(t,x)\in[0,T]\times \mathbb R^d} |\nabla u_i(t,x)| \leq \frac12.
\]
\end{itemize}
\end{proposition}

\begin{proof}
 For simplicity of notation we drop the subscript $i$ in the rest of the proof.
 We know that $u \in C_T \mathcal C^{1+\beta}$ by Remark \ref{rm:PDEexistence}.
 
 \emph{Item (i)} 
 By \eqref{eq:gamma12} we have
\[
\sup_{(t,x)\in [0,T]\times \mathbb R^d} |u(t,x)| \leq  \sup_{t\in [0,T]} \|u(t)\|_{1+\beta} 
=  \|u\| _{ C_T\mathcal C^{1+\beta}  } <\infty.
\]

\emph{Item (ii)}
By  \eqref{eq:gamma01} we have 
\[
\sup_{(t,x)\in [0,T]\times \mathbb R^d} |\nabla u(t,x)| \leq  \sup_{t\in [0,T]} \|\nabla u(t)\|_{\beta} .
\]  
Assume now (we will show it below) that the unique solution $u$ of \eqref{eq:PDEu} is also a solution of the integral equation
\begin{equation}\label{eq:ulambda}
u(t) = \int_t^T e^{-\lambda(s-t)} P_{s-t} \left(\nabla u(s) b(s) \right)  \di s  - \int_t^T e^{-\lambda(s-t)} P_{s-t} b(s)    \di s.
\end{equation}
From \eqref{eq:ulambda} we take the gradient on both sides and calculate its norm in $\mathcal C^\beta$. We use Schauder's estimates \eqref{eq:Pt},  Bernstein's inequality \eqref{eq:nabla}, and the fact that $\nabla u(s) b(s),  b(s)  \in \mathcal C^{-\beta+\eps} $ by pointwise product \eqref{eq:bony}
to get
\begin{align*}
\|\nabla u(t)\|_\beta \leq & \int_t^T\| \nabla( e^{-\lambda(s-t)} P_{s-t} \left(\nabla u(s) b(s) \right)) \|_{\beta}  \di s\\
 &+ \int_t^T \| \nabla (e^{-\lambda(s-t)} P_{s-t} b(s) )\|_{\beta}   \di s\\
\leq & c \int_t^T ( \|   e^{-\lambda(s-t)} P_{s-t} \left(\nabla u(s) b(s) \right)\|_{\beta+1}  + \| e^{-\lambda(s-t)} P_{s-t} b(s) \|_{\beta+1} )  \di s\\
\leq & c \int_t^T   e^{-\lambda(s-t)} ({s-t})^{-\frac{1+2\beta-\eps}{2}} \left(\| \nabla u(s) \|_{\beta} +1 \right) \|b(s)\|_{-\beta+\eps}  \di s\\
\leq & c \int_t^T   e^{-\lambda(s-t)} ({s-t})^{-\frac{1+2\beta-\eps}{2}} \di s  (1+ \sup_{s\in[0,T]}\| \nabla u(s) \|_{\beta}) \|b\|_{C_T\mathcal C^{-\beta+\eps}} ,
\end{align*}
where $c$ varies from line to line but it depends only on $\beta$ and $\eps$.
Since $\theta:= \frac{1+2\beta-\eps}{2}<1$ by assumption, the integral is bounded from above by $\Gamma(1-\theta) \lambda^{\theta-1}$ by a change of variable $\tilde s= \lambda (s-t)$ and using the definition of the Gamma function \eqref{eq:defGamma}. We get
\[
\sup_{t\in[0,T]}\|\nabla u(t)\|_\beta \leq  c\Gamma(1-\theta) \lambda^{\theta-1} (1+ \sup_{t\in[0,T]}\|\nabla u(t)\|_\beta ) \|b\|_{C_T\mathcal C^{-\beta+\eps}},
\]
that is
\[
\sup_{t\in[0,T]}\|\nabla u(t)\|_\beta (1-c\lambda^{\theta-1}\Gamma(1-\theta)\|b\|_{C_T\mathcal C^{-\beta+\eps}} ) \leq c \Gamma(1-\theta) \lambda^{\theta-1} \|b\|_{C_T\mathcal C^{-\beta+\eps}} 
\]
and choosing $\lambda$ according to \eqref{eq:lambda} with $C(\beta, \eps) = 3c\Gamma(1-\theta)$  we have   
\[
\sup_{t\in[0,T]}\|\nabla u(t)\|_\beta \leq \frac{ c\Gamma(1-\theta) \lambda^{\theta-1}\|b\|_{C_T \mathcal C^{-\beta+\eps}}} {1- c \Gamma(1-\theta) \lambda^{\theta-1}\|b\|_{C_T\mathcal C^{-\beta+\eps} }}\leq  \frac12, 
\]
as wanted.

It is left to prove that \eqref{eq:ulambda} holds. We can multiply both sides of  \eqref{eq:ulambda} by $e^{-\lambda t }$ to obtain
\[
e^{-\lambda t } u(t) = \int_t^T e^{-\lambda s } P_{s-t} \left(\nabla u(s) b(s) \right)  \di s  - \int_t^T e^{-\lambda s} P_{s-t} b(s)    \di s.
\]
Setting $\hat b(s):= e^{-\lambda s}b(s)$  we observe that the equation above writes
\[
e^{-\lambda t}  u(t) = \int_t^T  P_{s-t} \left(\nabla e^{-\lambda s} u(s) b(s) \right)  \di s  - \int_t^T P_{s-t} \hat  b(s)    \di s,
\] 
which is the mild form of the PDE (recall that mild and weak solutions are equivalent in $C_T D \mathcal C^{\beta}$ by Proposition \ref{pr:mildweak})
\begin{equation}\label{eq:vlambda}
\left\{
\begin{array}{l}
\partial_t v + \frac12\Delta v + \nabla v \, b = -\hat b\\
v(T) =0,
\end{array}
\right.
\end{equation}
where  $v(t):=e^{-\lambda t } u(t) $. 
Therefore to show that \eqref{eq:ulambda} holds it is enough to show that if $u$ is a weak solution of \eqref{eq:PDEu}, 
then $v(t)=e^{-\lambda t } u(t) $ is a weak solution of \eqref{eq:vlambda}. For $u$ weak solution of \eqref{eq:PDEu} then $u \in C^1([0,T]; \mathcal S')$ and \eqref{eq:vlambda} readily holds by time-differentiation. Moreover $v\in C_T D \mathcal C^{\beta}$ since $u\in C_T D \mathcal C^{\beta}$.
\end{proof}

Next we consider another special case of PDE \eqref{eq:PDE}.
Let us define the  vector-valued function $\phi:\R^d \to \R^d$  as  
\begin{equation}\label{eq:phi}
  \phi(t,x): =  u(t,x) + x,
\end{equation} 
where $u = (u_1,  \ldots, u_d)^\top$ and  $u_i$ is the  solution of \eqref{eq:PDEu}, unique in the sense of Remark \ref{rm:PDEexistence}, for $i=1, \ldots, d$.
We define $\phi$ as a column vector. 

\begin{theorem}\label{thm:PDEphi}
Each component $\phi_i$, for $i=1, \ldots, d$, of
the function  $\phi$ defined in \eqref{eq:phi} is the unique solution of 
\begin{equation}\label{eq:PDEphi}
\left\{
\begin{array}{l}
\mathcal L \phi_i = \lambda (\phi_i - \text{id}_i)\\
\phi_i(T) = \text{id}_i
\end{array}
\right.
\end{equation} 
in $C_T D\mathcal C^{\beta}$.
\end{theorem}

\begin{proof}
Using the linearity of the PDEs for $u_i$ and id$_i$ (see Corollary \ref{cor:Lid=b} and Remark \ref{rm:PDEexistence}) it is easy to check that each component  $\phi_i$, for $i=1, \ldots d$ solves \eqref{eq:PDEphi}.
 By Theorem \ref{thm:PDElambda0} item (i) we also have that $\phi_i$ is the unique solution of \eqref{eq:PDEphi}.
\end{proof}

\begin{proposition}\label{pr:phiunique}
Let $\phi$ be given by \eqref{eq:phi}. Then $\phi\in \mathcal D_{\mathcal L}^0$ and the time-derivative $\dot \phi_i$ is  in $C_T\mathcal C^{(-\beta)-}$ for all $i=1, \ldots d$.
\end{proposition}

\begin{proof}
In this proof we drop the subscript $i$ for ease of writing. 

By Theorem \ref{thm:PDEphi} and Remark \ref{rm:AC} we have  $\phi \in {\mathcal D}_{\mathcal L}^0$. 
Using  \eqref{eq:PDEphi} we get $\mathcal L \phi = \lambda (\phi-\text{id}) $ with $\phi(T) = \text{id}$, therefore concerning
the time-derivative $\dot \phi$ we have
\[
\int_0^t \dot \phi(s, \cdot) \di s = - \int_0^t \frac12 \Delta \phi(s, \cdot) \di s - \int_0^t \nabla \phi(s, \cdot) \, b (s, \cdot) \di s  +  \int_0^t \lambda u(s, \cdot) \di s.
\]
Since $\phi \in C_T \mathcal C^{(2-\beta)-}$ by Remark \ref{rm:PDElambda0}, we 
have $ \Delta \phi \in C_T \mathcal C^{(-\beta)-}$ and   $\nabla \phi \in C_T \mathcal C^{(1-\beta)-}$. Moreover 
$b\in C_T \mathcal C^{(-\beta)+}$, so $\nabla \phi \, b \in C_T \mathcal C^{(-\beta)+}$  by \eqref{eq:bony}, and $u \in C_T\mathcal C^{(2-\beta)-} $.
Thus    $\dot \phi \in C_T \mathcal C^{(-\beta)-}$.
\end{proof}

In the following proposition  we show that $\phi$ enjoys other  useful properties when $\lambda$ is large enough.
\begin{proposition}\label{pr:PDEphi}
 Let $\phi$ be given by \eqref{eq:phi}. 
\begin{itemize}
 \item[(i)] We have $\phi\in C^{0,1}$ and  $\nabla \phi\in C_T \mathcal C^{(1-\beta)-}$. In particular $\nabla \phi$ is uniformly bounded.
 \item[(ii)] For $\lambda$ as in Proposition \ref{pr:PDEproperties} we have that $\phi(t, \cdot)$ is invertible for all $t\in[0,T]$, with the (space-)inverse denoted by 
 \begin{equation}\label{eq:psi}
 \psi:= \phi^{-1}(t, \cdot).
 \end{equation} 
 Moreover $\psi\in C^{0,1}$,  $\nabla \psi$ is uniformly bounded  and  $\nabla \psi (t, \cdot) \in \mathcal C^{(1-\beta)-}$ for all $t \in  [0,T]$ and $\sup_{t\in[0,T]} \| \nabla \psi (t, \cdot)\|_{1-\alpha}<\infty$ for all $\alpha<\beta$.
 \end{itemize}
\end{proposition}

\begin{proof} 
{\em Item (i).} The fact that $\phi\in C^{0,1}$ follows from the fact that both id and $u$ are in $C^{0,1}$, since $u\in C_T\mathcal C^{(2-\beta)-}$ by Theorem \ref{thm:PDElambda0}. By the same regularity property of $u$ we also have $\nabla \phi \in C_T \mathcal C^{(1-\beta)-}$.

{\em Item (ii).} 
To show that $\phi(t, \cdot)$ is invertible one can proceed like in the proof of \cite[Lemma 22]{flandoli_et.al14}. This proof uses the fact that $|\nabla u(t,x)|\leq \frac12$ for $\lambda$ satisfying \eqref{eq:lambda} from  Proposition \ref{pr:PDEproperties}.
We can also easily see that $ \psi \in C^{0,1}$. Indeed $\nabla \phi$ is non-degenerate,  $\nabla \psi = \nabla\phi(\psi)^{-1}$ so that $(t,x) \mapsto \nabla\psi(t, \cdot) $ is continuous since $\phi\in C^{0,1}$ and $\psi \in C^{0,1}$. 
Here the superscript $-1$ denotes the matrix inverse.  
Finally we prove that   $\nabla  \psi (t, \cdot)  \in \mathcal C^{(1-\beta)-}$ for all  $t\in[0,T]$. We drop the time variable by ease of notation.  
We notice that $|\nabla \phi|$ is lower bounded  by $\frac12$ because $\nabla \phi = \nabla u  +\text{id}$, hence $|(\nabla \phi)^{-1}|$ is bounded by some constant $C$ independent of time and so  $|\nabla \psi|$ is bounded, where $|\cdot|$ denotes the Frobenious norm. Therefore $\psi$ is  Lipschitz. 
Using the fact that $\nabla \phi \in C_T \mathcal C^{(1-\beta)-}$ ,  $|\nabla \phi^{-1}|$ is bounded and that  $\psi$ is  Lipschitz, we have for $y,z\in \R^d$
\begin{align*}
|\nabla \psi(y) - \nabla\psi(z)| 
&= |\nabla \phi(\psi(y)) ^{-1} - \nabla \phi(\psi(z))^{-1} | \\
&= |\nabla \phi(\psi(z)) )^{-1} \left( \nabla \phi(\psi(z))- \nabla \phi(\psi(y)) \right) \nabla \phi(\psi(y))^{-1} | \\
&\leq  |\nabla \phi(\psi(z)) )^{-1} | \, | \nabla \phi(\psi(z))- \nabla \phi(\psi(y))| \, | \nabla \phi(\psi(y))^{-1} | \\
& \leq C | \nabla \phi(\psi(z))- \nabla \phi(\psi(y))|\\
&\leq C | \psi(z)-\psi(y)|^{1-\beta-\nu}\\
&\leq C | z-y|^{1-\beta-\nu},
\end{align*} 
for any $\nu >0$, where we recall that  $C$ does not depend on time.
\end{proof}

We now state and prove a continuity result for  PDEs with bounded or unbounded solutions. 
\begin{lemma}\label{lm:continuityv}
Let  Assumption \ref{ass:param-b} hold. Let $\lambda >0$ be fixed. 
Let $b^n$ be a sequence converging to $ b$ in $C_T \mathcal C^{-\beta} $,  $g^n \to g$  in $C_T \mathcal C^{-\beta} $.  Then 
\begin{itemize}
\item [(i)] if $v_T^n \to v_T $ in $D\mathcal C^{(1-\beta)-} $ then $v^n\to v$ in $C_T D\mathcal C^{(1-\beta)-}$;
\item[(ii)]  if $v_T^n \to v_T $ in $\mathcal C^{(2-\beta)-} $ then $v^n\to v$ in $C_T \mathcal C^{(2-\beta)-}$,
\end{itemize}
where $v^n$ is the unique solution of \eqref{eq:PDE} with $b$ replaced by $b^n$, $g$ replaced by $g^n$ and $v_T$ replaced by $v_T^n$.\\
In particular $\nabla v^n \to \nabla v$ in $C_T \mathcal C^{(1-\beta)-}$.
\end{lemma}
\begin{proof}
We show both items at the same time.

To show that $v^n \to v$ in $ C_T D\mathcal C^{(1-\beta)-}$ (resp.\ in $ C_T \mathcal C^{(2-\beta)-}$) we have to show that for all $\alpha<1-\beta$ such that $v^n \to v$ in $ C_T D\mathcal C^{\alpha}$ (resp.\ in $ C_T \mathcal C^{1+\alpha}$). Since $v_T^n\to v_T $ in $D\mathcal C^{(1-\beta)-}$  (resp.\  in $\mathcal C^{(2-\beta)-} $) for all $\alpha<1-\beta$ such that $v_T^n\to v_T$ in $D\mathcal C^{\alpha}$ (resp.\ in $\mathcal C^{1+\alpha} $ ), we  fix any $\alpha<1-\beta$. 
We show that $\|v^n - v\|_{ C_T D\mathcal C^{\alpha}}^{(\rho)} \to 0 $ (resp. $\|v^n - v\|_{ C_T \mathcal C^{1+\alpha}}^{(\rho)} \to 0 $) as $n\to \infty$, where the superscript $(\rho)$ denotes the $\rho$-equivalent norm introduced in Section \ref{sc:prelim}. Using the definition of mild solution we have
\begin{align*}
 &v^n (t) -v(t) = P_{T-t} ( v^n_T-v_T)\\
& + \int_t^T P_{s-t} \big( \nabla v^n(s)b^n(s) +  \nabla v(s)b^n(s) -  \nabla v(s)b^n(s)   -\nabla v(s)b(s)  \big) \di s\\
& + \int_t^T P_{s-t} ( g^n(s) -g(s)  ) \di s + \lambda\int_t^T P_{s-t} ( v^n(s) -v(s)  ) \di s  . 
\end{align*}
Let us calculate the  $\|\cdot\|_{D\mathcal C^\alpha} $-norm (resp.\  $\|\cdot\|_{\mathcal C^{1+\alpha}} $-norm)  of the quantity above:
\begin{align*}
\|v^n & -v\|_{ C_T D\mathcal C^{\alpha}}^{(\rho)} 
= \sup_{0\leq t\leq T} e^{-\rho(T-t)} \| v^n(t)-v(t) \|_{D\mathcal C^\alpha}\\
\leq&   \sup_{0\leq t\leq T} e^{-\rho(T-t)} \| P_{T-t} ( v^n_T-v_T)\|_{D\mathcal C^\alpha} \\
&+  \|\int_t^T   P_{s-t} \big(  (\nabla v^n(s)-\nabla v(s) ) b^n(s) \big)  \di s\|_{C_T D\mathcal C^\alpha}^{(\rho)}\\
&+   \|\int_t^T  P_{s-t} \big(  \nabla v(s) (b^n(s) -b(s)) \big)  \di s\|_{C_T D\mathcal C^\alpha}^{(\rho)}\\
&+  \|\int_\cdot^T  P_{s-\cdot}   (g^n(s)-g(s) )  \di s\|_{C_T D\mathcal C^\alpha}^{(\rho)}\\
&+ \lambda \sup_{0\leq t\leq T} e^{-\rho(T-t)} \|\int_t^T  P_{s-t}   (v^n(s)-v(s) )  \di s\|_{D\mathcal C^\alpha}\\
&=: B_1 + B_2 + B_3+ B_4+ B_5,
\end{align*}
(respectively $ \|v^n -v\|_{ C_T \mathcal C^{1+\alpha}}^{(\rho)} =: B_1 + B_2 + B_3+ B_4+ B_5$, where the norm in $D\mathcal C^{\alpha}$ is substituted by the one in $ \mathcal C^{1+\alpha}$).\\
The terms $B_1$ and $B_5$ are bounded  using Lemma \ref{lm:PtD} (resp.\ \eqref{eq:Pt} with $\theta=0$) to get
\begin{align*}
&B_1  \leq  \sup_{0\leq t\leq T} \| P_t(  v^n_T-v_T)\|_{D\mathcal C^\alpha}  \leq  c \|   v^n_T-v_T\|_{D\mathcal C^{\alpha}},\\
 &B_5 \leq \lambda  \int_t^T e^{-\rho(s-t)} e^{-\rho(T-s)}\|v^n(s)-v(s)\|_{ D\mathcal C^\alpha} \di s  \leq c\rho^{-1}   \|v^n-v\|^{(\rho)}_{C_T D\mathcal C^\alpha},
 \end{align*}
(respectively similar estimates where the norm in $D\mathcal C^{\alpha}$ is substituted by the one in $ \mathcal C^{1+\alpha}$).\\
For $B_2$ and $B_3$ we apply  Lemma \ref{lm:Ptnablawh} and \eqref{eq:bonytrho} twice and for the term $B_4$ we only apply  Lemma \ref{lm:Ptnablawh}  to get
\begin{align*}
&B_2  \leq c  \|b^n\|_{C_T \mathcal C^{-\beta} } \|  \nabla ( v^n-v)\|^{(\rho)}_{C_T \mathcal C^{\alpha}} \rho^{\frac{\alpha+\beta-1}2}  \leq c  \|b^n\|_{C_T \mathcal C^{-\beta} } \|   v^n-v\|^{(\rho)}_{C_T D \mathcal C^{\alpha}} \rho^{\frac{\alpha+\beta-1}2},\\
&B_3 \leq c \|b^n-b\|_{C_T \mathcal C^{-\beta} } \| \nabla  v\|^{(\rho)}_{C_T \mathcal C^{\alpha}} \rho^{\frac{\alpha+\beta-1}2} \leq c \|b^n-b\|_{C_T \mathcal C^{-\beta} } \|  v\|^{(\rho)}_{C_T D\mathcal C^{\alpha}} \rho^{\frac{\alpha+\beta-1}2},\\
&B_4 \leq c \|g^n-g\|^{(\rho)}_{C_T \mathcal C^{-\beta}} \rho^{\frac{\alpha+\beta-1}2},
 \end{align*}
 (respectively similar estimates where the norm in $D\mathcal C^{\alpha}$ is substituted by the one in $ \mathcal C^{1+\alpha}$). Thus we have
\begin{align*}
\|v^n  -v\|_{ C_T D\mathcal C^{\alpha}}^{(\rho)} &\leq   c \|   v^n_T-v_T\|_{D\mathcal C^{\alpha}} \\
&+  c  \|b^n\|_{C_T \mathcal C^{-\beta} } \|   v^n-v\|^{(\rho)}_{C_T D \mathcal C^{\alpha}} \rho^{\frac{\alpha+\beta-1}2} \\
&+ 
 c \|b^n-b\|_{C_T \mathcal C^{-\beta} } \|  v\|^{(\rho)}_{C_T D\mathcal C^{\alpha}} \rho^{\frac{\alpha+\beta-1}2}\\
&+   c \|g^n-g\|^{(\rho)}_{C_T \mathcal C^{-\beta}} \rho^{\frac{\alpha+\beta-1}2} + 
 c\rho^{-1}   \|v^n-v\|^{(\rho)}_{C_T D\mathcal C^\alpha},
 \end{align*}
  (respectively similar estimates where the norm in $D\mathcal C^{\alpha}$ is substituted by the one in $ \mathcal C^{1+\alpha}$).
  
Similarly to \eqref{eq:rhob} but replacing $\|b\|_{C_T \mathcal C^{-\beta}}$ with $\sup_n \|b^n\|_{ C_T \mathcal C^{-\beta}}$, we choose $\rho\geq 1$ such that 
\[
c(1 + \sup_n\|b^n\|_{C_T \mathcal C^{-\beta} }) \rho^{\frac{\alpha+\beta-1}2} \leq \frac12,
\]
 so that combining the estimates above and  moving to the left-hand side the terms involving $v^n-v$ we get
 \begin{align*}
\frac12  \|v^n  -v\|_{ C_T D\mathcal C^{\alpha}}^{(\rho)} \leq  &c \|   v^n_T-v_T\|_{D\mathcal C^{\alpha}} + c\|g^n-g\|_{C_T \mathcal C^{-\beta}}^{(\rho)}\rho^{\frac{\alpha+\beta-1}2}\\
& + c\|b^n-b\|_{C_T \mathcal C^{-\beta} } \|  v\|^{(\rho)}_{C_T D\mathcal C^{\alpha}} \rho^{\frac{\alpha+\beta-1}2},
 \end{align*}
 (respectively similar estimates where the norm in $D\mathcal C^{\alpha}$ is substituted by the one in $ \mathcal C^{1+\alpha}$).
 The proof is concluded. 
\end{proof}

\begin{remark}\label{rm:continuityv}
Following the proof of Lemma  \ref{lm:continuityv}, it is easy to see that a slightly weaker convergence   remains valid under slightly weaker assumptions, namely
\begin{itemize}
\item [(i)] if $v_T^n \to v_T $ in $D\mathcal C^{\beta+} $ then $v^n\to v$ in $C_T D\mathcal C^{\beta+}$;
\item[(ii)]  if $v_T^n \to v_T $ in $\mathcal C^{(1+\beta)+} $ then $v^n\to v$ in $C_T \mathcal C^{(1+\beta)+}$.
\end{itemize}
In particular $\nabla v^n \to \nabla v$ in $C_T \mathcal C^{\beta+}$.
\end{remark}

\begin{lemma}\label{lm:continuityphi}
 Let $b^n \to b$  in $C_T \mathcal C^{-\beta}$. Let $\lambda$ be such that 
\begin{equation}\label{eq:lambda2}
\lambda^{1-\theta} = C(\beta,\eps) \max\{ \sup_n \|b^n\|_{C_T \mathcal C^{-\beta+\eps}}, \|b\|_{C_T \mathcal C^{-\beta+\eps}} \}
\end{equation}
with $\theta:= \tfrac{1+2\beta-\eps}2<1$ and $C(\beta,\eps)$ chosen according to Proposition \ref{pr:PDEproperties} item (ii). 
Let $\phi^n$  be defined as in \eqref{eq:phi} but with $b$ replaced by $b^n$ and let  $\psi^n$ be the (space-)inverse of $\phi^n$ as in \eqref{eq:psi}. Then we have  
\begin{itemize}
\item[(i)] $u^n \to u, \nabla u^n \to \nabla u, \phi^n \to \phi$ and $\psi^n \to \psi$ uniformly on $[0,T]\times \R^d$;
\item[(ii)]  $\|\nabla \phi^n\|_\infty$ and $|\phi^n(0,0)|$ are uniformly bounded in $n$.
\end{itemize}
\end{lemma}
\begin{proof}
We choose $\lambda $ according to \eqref{eq:lambda2}  as done in \eqref{eq:lambda}. This implies
 \begin{equation}\label{eq:supnablau}
\sup_{(t,x)\in[0,T]\times\R^d} |\nabla u^n(t,x)| \leq \frac12
 \end{equation}
 by Proposition 
\ref{pr:PDEproperties} part (ii).  

{\em Item (i)} 
By Lemma \ref{lm:continuityv} part (ii) we have   $u_{n}\rightarrow u$ in $C_T \mathcal C^{\alpha+1}
$ thus  $u_{n}\rightarrow u$ and $\nabla u_{n}\rightarrow\nabla u$, uniformly on
$\left[  0,T\right]  \times\mathbb{R}^{d}$.
 Since $\phi_{n}-\phi
=u_{n}-u$, then also $\phi_{n}\rightarrow\phi$ uniformly on
$\left[  0,T\right]  \times\mathbb{R}^{d}$. 

The rest of the proof follows the same ideas  of \cite[Lemma 24, part (iii)]{flandoli_et.al14}. We recall the basic elements of the  proof for ease of reading. 
Let us prove the uniform
convergence of $\psi_n$ to $\psi$. Given $y\in\mathbb{R}^{d}$, we know that for every $t\in\left[  0,T\right]  $
and $n\in\mathbb{N}$ there exist $x\left(  t\right)  ,x_{n}\left(  t\right)
\in\mathbb{R}^{d}$ such that
\begin{align*}
x\left(  t\right)  +u\left(  t,x\left(  t\right)  \right)    & =y\\
x_{n}\left(  t\right)  +u_{n}\left(  t,x_{n}\left(  t\right)  \right)    & =y
\end{align*}
and we have called $x\left(  t\right)  $ and $x_{n}\left(  t\right)  $ by
$\psi\left(  t,y\right)  $ and $\psi_n\left(  t,y \right)  $
respectively. Then from \eqref{eq:supnablau}  we get
\begin{align*}
\left\vert x_{n}\left(  t\right)  -x\left(  t\right)  \right\vert  
\leq &\sup_{(t,x) \in [0,T]\times \mathbb R^d}\left\vert \nabla u_{n}(t,x)\right\vert \left\vert x_{n}\left(
t\right)  -x\left(  t\right)  \right\vert \\
&+\sup_{(t,x) \in [0,T]\times \mathbb R^d}\left\vert u_{n}(t,x)-u(t,x)\right\vert\\
\Rightarrow\left\vert x_{n}\left(  t\right)  -x\left(  t\right)  \right\vert   \leq&2 \sup_{(t,x) \in [0,T]\times \mathbb R^d} \left\vert u_{n}(t,x)-u(t,x)\right\vert,
\end{align*}
 namely%
\[
\left\vert \psi_n\left(  t,y\right)  -\psi\left(
t,y\right)  \right\vert \leq2 \sup_{(t,x) \in [0,T]\times \mathbb R^d} \left\vert u_{n}(t,x)-u(t,x)\right\vert,
\]
which implies that $\psi_n\rightarrow\psi$ uniformly on
$\left[  0,T\right]  \times\mathbb{R}^{d}$.

{\em Item (ii)}
 To show that $ \|\nabla \phi^n\|_\infty $ is bounded uniformly in $n$ we simply observe that $\nabla \phi^n(t,x)= \text{id} + \nabla u^n(t, x)$ and use \eqref{eq:supnablau}.

To prove that $|\phi^n(0,0)|= |u^n(0,0)|$ is uniformly bounded we observe that $u^n \to u$ in   $C_TD\mathcal C^{(1-\beta)-}$ by Lemma \ref{lm:continuityv} part (i), hence there exists $\alpha <1-\beta$  such that  $u^n \to u$ in   $C_TD\mathcal C^{\alpha}$ and so 
\[
\sup_{n\geq1} |u^n(0,0)| \leq c \sup_{n\geq1} \|u^n\|_{C_T D\mathcal C^{\alpha} },
\]
which concludes the proof. 
\end{proof}

\section{On some separable Besov-H\"older type spaces}\label{sc:separable}

In the companion paper \cite{issoglio_russoMPb} we  use a special class of PDEs like \eqref{eq:parabolicPDE} for some applications in stochastic analysis. In particular, the PDE plays a role   in the formulation of the martingale problem for stochastic differential equations with distributional drifts $b$.  For more details on the latter, see  \cite[Section 4]{issoglio_russoMPb}. The class of PDEs that we  use in \cite{issoglio_russoMPb} are PDEs of the form
$\mathcal L f = g,$ where the element $g$ is a function (instead of a distribution) that, most importantly,  lives in a space which is separable. The spaces $C_T \mathcal C^{0+}$ would be the natural choice since  it contains only functions,  but it is not separable. It would be separable  if  one restricted them to functions with compact support, however the class   $C_T \mathcal C_c^{0+}$ of functions in  $C_T \mathcal C^{0+}$  with compact support  is not closed under the topology of $C_T \mathcal C^{0+}$ and not rich enough for our purpose. Thus here we introduce and investigate a further class of function spaces, namely the closure of $C_T \mathcal C_c^{0+}$ with respect to the topology of  $C_T \mathcal C^{0+}$. These spaces turn out to be separable and rich enough to be  used in our application to stochastic analysis.
In this section, we prove some useful results about these space, most importantly separability.

\begin{lemma}\label{lm:lemmino}
Let $f$ be a Schwartz distribution with compact support. 
We have  $f \ast p_t \in \mathcal S$ for all $t>0$. 
\end{lemma}
\begin{proof}
We will show that the Fourier transform  $ \mathcal F( p_t \ast f)$   of $ p_t \ast f$ is in $\mathcal S$. 
Since $f$ is a  compactly supported  Schwartz distribution we apply 
 \cite[Theorem 26, page 91]{schwartz} to write    $ f$   as the finite sum $\sum_\nu \partial^\nu h$ with $h$ some continuous function with compact support. By linearity it is enough to show that $\mathcal F (\partial^\nu h \ast p_t) \in \mathcal S$, where $h$ some continuous function with compact support. In this case we have   
 \[
\mathcal F( \partial^\nu h  \ast p_t) =    \mathcal F(h \ast \partial^\nu  p_t)  = \mathcal F(h ) \mathcal F( \partial^\nu  p_t),
 \]
and this belongs to $\mathcal S$ since $\mathcal F \partial^\nu p_t \in \mathcal S$ and $ \mathcal F h \in C_b^\infty$ by an easy calculation.    
\end{proof}

We denote by $C_c=C_c(\R^d)$ the space of $\R^d$-valued continuous functions with compact support.
For $\gamma\geq0$ we denote by $\mathcal C_c^{\gamma}=\mathcal C_c^{\gamma}(\R^d)$ the space of elements in $\mathcal C^{\gamma}$ with compact support. Similarly when $\gamma$ is replaced by $\gamma+$ or $\gamma-$, for $\gamma\geq 0$. When defining the domain of the martingale problem  we will work with spaces of functions which are the limit of functions with compact support, so that they are Banach space. More precisely, let us denote by $\bar{\mathcal C}_c^\gamma  = \bar{\mathcal C}_c^\gamma (\R^d)$ the space  
\[
\bar{\mathcal C}_c^\gamma := \{f \in {\mathcal C}^\gamma \text{ such that } \exists (f_n)_n \subset {\mathcal C}_c^\gamma \text{ and } f_n \to f \text{ in }  {\mathcal C}^\gamma\}.
\]
As above we denote the inductive space and intersection space as
\[
\bar {\mathcal C}_c^{\gamma+}:= \cup_{\alpha >\gamma} \bar{\mathcal C}_c^{\alpha} ,  \qquad  
\bar {\mathcal C}_c^{\gamma-}:= \cap_{\alpha <\gamma} \bar{\mathcal C}_c^{\alpha}.
\]
We also introduce the space $C_T \bar {\mathcal C}_c^{\gamma+}$ and observe that  $f \in C_T \bar {\mathcal C}_c^{\gamma+}$ if and only if there exists $\alpha>\gamma $ such that $f\in C_T \bar{\mathcal C}_c^{\alpha}$, 
by \cite[Remark B.1 part (ii)]{issoglio_russoMK}.

We will state and prove several useful properties of such spaces. Let us start by showing that $C_T\bar{\mathcal C}_c^\gamma $ is an algebra.

\begin{proposition}
The space $C_T \bar{\mathcal C}_c^\gamma $ is an algebra for $\gamma\in(0,1)$. 
\end{proposition}
\begin{proof}
  Let $f, g\in C_T\bar{\mathcal C}_c^\gamma $.
  By \cite[Remark  B.1]{issoglio_russoMK}, we know that there exists a sequence $(f_n)_n \subset C_T \mathcal C_c^\gamma $ (resp.\ $(g_n)_n$) such that $f_n \to f$  (resp.\ $g_n \to g$) in $C_T \mathcal C^\gamma$. Clearly $f_n g_n \in C_T \mathcal C_c^\gamma$ so it remains to  show that  $f_ng_n \to fg$ in $C_T\mathcal C^\gamma$. 
We have $f_n g_n -fg = (f_n - f)g_n + f(g_n-g)$ so it is enough to show that $(f_n - f)g_n  \to 0$ and  $f(g_n-g)\to 0$ in $C_T \mathcal C^\gamma$. We show the first term only, as the second can be handled the same (but easier). Using the norm \eqref{eq:gamma01} we need to bound two terms. 
The first one is $\sup_{t\in[0,T]}\|(f_n(t, \cdot) - f(t, \cdot))g_n(t, \cdot)\|_\infty$ and it clearly converges to 0
 by assumptions on $f_n, g_n$. As for the H\"older seminorm      for all $t\in[0,T]$ we have
 \begin{align*}  
 &|(f_n - f)(t, x)  g_n(t, x) - (f_n- f)(t, y)g_n(t, y)| \\
 &\leq |[(f_n - f)(t, x) - (f_n- f)(t, y)]g_n(t, x)| 
\\ 
& + | (f_n- f)(t, y)[g_n(t, x)- g_n(t, y)]|\\
&\leq \|f_n - f\|_{C_T \mathcal C^\gamma} |x-y|^{\gamma} \sup_{t,x} |g_n(t, x)|\\
&+\sup_{t,x} | (f_n- f)(t, x)| \|g_n\|_{C_T \mathcal C^\gamma} |x-y|^{\gamma}.
\end{align*}
Using this we conclude that $$\sup_{t\in[0,T]} \sup_{x\neq y}\frac{|(f_n - f)(t, x)  g_n(t, x) - (f_n- f)(t, y)g_n(t, y)|}{|x-y|^{\gamma}} \to 0,$$
by the fact that $f_n \to f$ uniformly and $ \|f_n - f\|_{C_T \mathcal C^\gamma}$ and $\|g_n\|_{C_T \mathcal C^\gamma} $ are bounded. 
\end{proof}

\begin{lemma}\label{lm:SCgamma}
We have \begin{equation}\label{eq:SSS}
\mathcal S \subset \bar{\mathcal C}_c^{\gamma+}
\end{equation}
 for $\gamma \in \R$.
 In particular, $\mathcal S$ is included in the
 closure $\bar C_c$ of the space of continuous functions with compact support  $C_c  $  with respect to the topology of uniform convergence. 
\end{lemma}

\begin{proof}
It is enough to show the claim for  every $\gamma\geq0$.
We only prove \eqref{eq:SSS} since the closure of the space of continuous functions with compact support  $C_c (\R^d)$  with respect to the topology of uniform convergence contains $\bar{\mathcal C}_c^{\gamma+}$.

Let $\chi: \mathbb R \to \R_+$ be a smooth function such that 
\[
\chi (x) =\left\{
\begin{array}{ll}
0 \quad &x\geq0\\
1 \quad &x\leq-1\\
\in (0,1) \quad  & x\in(-1,0). 
\end{array}
\right.
\]
We set  $\chi_n: \R^d \to \R$ as
$
\chi_n(x) : = \chi (|x| - (n+1)).
$
In particular 
\[
\chi_n (x) =\left\{
\begin{array}{ll}
0 \quad &|x|\geq n+1\\
1 \quad &|x|\leq n\\
\in(0,1) \quad  & \text{otherwise}.
\end{array}
\right.
\]
Let $f \in\mathcal S$. We  set $f_n( x) := f( x) \chi_n(x) $. Clearly $f_n \in \mathcal C_c^{\gamma+}$.

{\em Step 1.}
For any multi-index $m$  we first  show that $D^m f_n \to D^m f$ uniformly.

Notice that $D^m (f_n-f) = D^m( f (1-\chi_n)) $ is  a finite sum of terms of the form $D^l f D^{k} (1-\chi_n) $ for some finite $|l|,|k| \leq|m|$. 
One can show that   
$\sup_x |D^k (1-\chi_n)(x)|\leq \|D^k  \chi \|_\infty$ by the definition of $\chi_n$. Let  $\eps>0$. Since  $f \in \mathcal S$ 
there  exists $n(\eps)$ such that for all $|x|>n(\eps)$ then $|D^l f(x)|<\eps$ for all $l$ such that  $|l|\leq |m|$. 
For $|x|>n(\eps)$ we have 
\[
|D^l f (x) D^{k} (1-\chi_n)(x) |\leq \|D^k  \chi \|_\infty \eps.
\]
This shows uniform convergence of $D^l f D^{k} (1-\chi_n) $ to 0, hence uniform convergence of $D^m (f_n-f)$ to zero. 

{\em Step 2.} 
Let $\alpha \in (0,1)$. For any multi-index $m$  it remains to show that 
\[
\sup_{|x-y|<1} \frac{|D^m (f(1-\chi_n))(x)- D^m (f(1-\chi_n))(y) |}{|x-y|^\alpha}
\]
converges to 0 as $n\to \infty $. We clearly  have that 
\[
\frac{|D^m (f(1-\chi_n))(x)- D^m (f(1-\chi_n))(y) |}{|x-y|^\alpha} \leq \|\nabla D^m (f(1-\chi_n)) \|_\infty |x-y|^{1-\alpha}
\]
by finite increments theorem, hence we reduce to Step 1.
\end{proof}


\begin{lemma}\label{lm:dense}
\begin{itemize}
\item[(i)]   For any $\gamma\in \R$ the space  ${\mathcal S}$ is dense in $\bar{\mathcal C}_c^{\gamma+}$.
\item[(ii)]  ${\mathcal S}$ is dense in $\bar{ C}_c$.
\end{itemize}
\end{lemma}

\begin{proof}
{\em Item (i)}
We observe  that $\mathcal S \subset  \bar{\mathcal C}_c^{\gamma+}$, see  Lemma \ref{lm:SCgamma}.
Let $\gamma \in \R$ and  $f \in \bar{\mathcal C}_c^{\gamma+}$. 
 By the definition of the space we can reduce to the case $f\in \mathcal C_c^{\gamma+}$. 
We mollify $f$ using the heat semigroup $P_\eps$, that is we consider $P_\eps f =  p_\eps \ast f$ where $p_\eps$ is the heat kernel.  By Lemma \ref{lm:lemmino} we have  $P_\eps f \in \mathcal S$. 
By \eqref{eq:Pt-I} we also have that  
$P_\eps f \to f $ in $\mathcal C^{\gamma+}.$

{\em Item (ii)} 
 The result follows from the fact that $P_\eps f \to f$ uniformly, for $f\in C_c$ and that $\mathcal S \subset \bar C_c$ by Lemma \ref{lm:SCgamma}.
\end{proof}

The next three lemmata will be used below  to prove that the spaces are separable.

\begin{lemma}\label{lm:BernsteinBanach}
Let $f:[0,1] \to B$ where $(B, \| \cdot \|)$ is a Banach space. Then the sequence $(f_n)_n$ defined by 
$f_n (t) := \sum_{j=0}^n f(\frac jn ) t^j (1-t)^{n-j}{ n \choose j}$ converges uniformly to $f$.
\end{lemma}
\begin{proof}
The polynomials $(f_n)_n$ are also know as Bernstein polynomials, often denoted by $B_n(f,t)$ that is
\begin{equation}\label{eq:bern_pol}
B_n(f,t) := f_n(t): = \sum_{j=0}^n f(\frac jn ) t^j (1-t)^{n-j}{ n \choose j}.
\end{equation}
Bernstein polynomials have the property that they can be expressed as expectations of suitable random variables, which is useful in the computations below. In particular,
let $U_1, \ldots, U_n \sim U(0,1)$ be independent uniform r.v.s and let 
$$ S_n(t) := \frac 1n \sum_{j=1}^n \mathbbm{1}_{[0,t)} (U_j) .$$
Since $S_n(t)$ is a binomial r.v. with parameter $n$ and $t$,
then clearly
\begin{equation}\label{eq:bernstein}
B_n(f, t) = \mathbb E[f (S_n(t))].
\end{equation}

Let $\eps>0$. Since $f$ is uniformly continuous, there exists  $\delta>0$ such that if $|t-s|\leq \delta $ then $\|f(t)-f(s)\|\leq  \eps$. Let $t\in[0,1]$, by \eqref{eq:bernstein} we have
\begin{align*}
  \Vert f_n(t) - f(t)\Vert  =& \Vert \mathbb E[f(S_n(t)) - f(t) ]\Vert\\
\leq & \mathbb E[\Vert f(S_n(t)) - f(t)\Vert \mathbbm 1_{\{|S_n(t) -t|\leq \delta\}} ]\\
&+\mathbb E[\Vert f(S_n(t)) - f(t)\Vert \mathbbm 1_{\{|S_n(t) -t|> \delta\}} ]\\
=:& I_1(t) +I_2(t).
\end{align*} 
Now 
\[I_1(t) \leq \eps \mathbb P (|S_n(t) -t|\leq \delta) \leq \eps.\]
Concerning $I_2(t),$
being $S_n(t)$ a binomial random variable with parameter $n$ and $t$,
$$\text{Var}(S_n(t)) =
\frac1n (t-t^2). $$
Using this and by Chebyshev inequality  we get
\begin{align*}
I_2(t) \leq & 2 \|f\|_\infty \mathbb P ( |S_n(t) -t|> \delta)\\
\leq &  2 \|f\|_\infty \frac{\text{Var}(S_n(t) -t)}{\delta^2}\\
\leq &  2 \|f\|_\infty \frac{(t -t^2)}{n\delta^2}.
\end{align*}
Now taking the supremum over $t\in [0,1]$ we get 
$
\sup_{t\in[0,1]} I_2(t) \leq
\frac12 \|f\|_\infty \frac{1}{n\delta^2}
$
and putting this together with the bound for $I_1(t)$ we obtain 
\[
\limsup_{n\to \infty} \sup_{t\in[0,1]}\|f_n(t) - f(t)\| \leq \limsup_{n\to \infty} \sup_{t\in[0,1]} (I_1(t)+I_2(t))\leq \eps.
\]
Since $\eps >0$ is arbitrary, the  proof is concluded. 
\end{proof}

\begin{lemma}\label{lm:CTEsep}
Let $E$ be an inductive space of the form  $E= \cup_{\alpha \in \N} E_\alpha$, with $E_\alpha$ Banach space. 
If $E$ is separable then  $C_T E$ is separable. 
\end{lemma}

\begin{proof}
Without loss of generality we choose $T=1$. Let $f\in C_TE$ and we consider the functions 
$$f_n (t) := \sum_{j=0}^n f (\frac jn) t^j (1-t)^{n-j}{ n \choose j}.$$
We now use the fact that $C_T E =  \cup_{\alpha\in \N} C_T E_\alpha$ by \cite[Remark B.1]{issoglio_russoMK}, where the space $C_T E_\alpha$ can be equipped with the norm $\sup_t \|f(t)\|_{E_\alpha}$. 
By this fact, there exists $\alpha$ such that $f\in C_T E_\alpha$, in particular $f_n \in C_T E_\alpha$ for all $n$. 
By Lemma \ref{lm:BernsteinBanach} $f_n$ converges to $f$ in $C_T E_\alpha$, which by the fact stated above implies it converges also  in $C_T E$. We have thus reduced our problem to polynomials of the form $\sum_{j=1}^n a_j t^j$ with $a_j \in E_\alpha$.
We conclude the proof  by using the fact that  $E$ is separable,  thus there exists a countable dense subset of $E$, say $\mathcal P$, so that every polynomial $\sum_{j=1}^n a_j t^j$  can be approached by a sequence of polynomials of the  type $\sum_{j=1}^n q_j t^j$  with $q_j \in \mathcal P$.
\end{proof}

\begin{lemma}\label{lm:Cbetasep}
\begin{itemize}
\item[(i)]    For any $\gamma\in \R$ the space  $\bar{\mathcal C}_c^{\gamma+} $ is separable.
\item[(ii)]   $\bar C_c  $   is separable.
\end{itemize}
\end{lemma}

\begin{proof}
 This follows from Lemma \ref{lm:dense}.
\end{proof}

\begin{corollary}\label{cor:CTCbetasep}
The space $C_T \bar{\mathcal C}^{\gamma+}_c $ is separable for any $\gamma\in\R$.
\end{corollary}

\begin{proof}
Notice that by definition $\bar{\mathcal C}^{\gamma}_c $ is a Banach space and the inductive space $\bar{\mathcal C}^{\gamma+}_c $ is separable by Lemma  \ref{lm:Cbetasep}, so we conclude  using Lemma \ref{lm:CTEsep}. 
\end{proof}

\bibliographystyle{plain}

\bibliography{../../../../BIBLIO_FILE/biblio}

\def\cprime{$'$}
\begin{thebibliography}{10}

\bibitem{bahouri}
H.~Bahouri, J-Y. Chemin, and R.~Danchin.
\newblock {\em Fourier Analysis and Nonlinear Partial Differential Equations}.
\newblock Springer, 2011.

\bibitem{bony}
J.-M. Bony.
\newblock Calcul symbolique et propagation des singularites pour les
  \'{e}quations aux d\'{e}riv\'{e}es partielles non lin\'{e}aires.
\newblock {\em Ann. Sci. Ec. Norm. Super.}, 14:209--246, 1981.

\bibitem{cannizzaro}
G.~{Cannizzaro} and K.~{Chouk}.
\newblock {Multidimensional SDEs with singular drift and universal construction
  of the polymer measure with white noise potential}.
\newblock {\em {Ann. Probab.}}, 46(3):1710--1763, 2018.

\bibitem{catellier_chouk}
R.~{Catellier} and K.~{Chouk}.
\newblock {Paracontrolled distributions and the 3-dimensional stochastic
  quantization equation}.
\newblock {\em {Ann. Probab.}}, 46(5):2621--2679, 2018.

\bibitem{menozzi}
P.-E. Chaudru~de Raynal and S.~Menozzi.
\newblock On {M}ultidimensional stable-driven {S}tochastic {D}ifferential
  {E}quations with {B}esov drift.
\newblock {\em Arxiv 2109.12263}, 2019.

\bibitem{flandoli_et.al14}
F.~Flandoli, E.~Issoglio, and F.~Russo.
\newblock Multidimensional {SDE}s with distributional coefficients.
\newblock {\em T. Am. Math. Soc.}, 369:1665--1688, 2017.

\bibitem{gubinelli_imkeller_perkowski}
M.~Gubinelli, P.~Imkeller, and N.~Perkowski.
\newblock Paracontrolled distributions and singular {PDE}s.
\newblock {\em Forum of Mathematics, Pi}, 3:75 pages, 2015.

\bibitem{issoglio13}
E.~Issoglio.
\newblock Transport equations with fractal noise - existence, uniqueness and
  regularity of the solution.
\newblock {\em J. Analysis and its App.}, 32(1):37--53, 2013.

\bibitem{issoglio19}
E.~Issoglio.
\newblock A non-linear parabolic {PDE} with a distributional coefficient and
  its applications to stochastic analysis.
\newblock {\em J. Differential Equations}, 267(10):5976--6003, 2019.

\bibitem{issoglio_russoMK}
E.~Issoglio and F.~Russo.
\newblock Mc{K}ean {SDE}s with singular coefficients.
\newblock {\em Annales de l'Institut Henri Poincar\'e. To appear}.
\newblock Arxiv 2107.14453, 2021.

\bibitem{issoglio_russoMPb}
E.~Issoglio and F.~Russo.
\newblock {SDE}s with singular coefficients: the martingale problem view and
  the stochastic dynamics view.
\newblock {\em Preprint Arxiv 2208.10799}, 2022.

\bibitem{lunardi95}
A.~Lunardi.
\newblock {\em Analytic semigroups and optimal regularity in parabolic
  problems}.
\newblock Progress in Nonlinear Differential Equations and their Applications,
  16. Birkh\"auser Verlag, Basel, 1995.

\bibitem{rudin}
W.~Rudin.
\newblock {\em Functional Analysis}.
\newblock Higher mathematics series. McGraw-Hill, 1973.

\bibitem{schwartz}
L.~Schwartz.
\newblock {\em Th{\'e}orie des distributions}.
\newblock Paris: Hermann, nouveau tirage edition, 1998.

\end{thebibliography}

\end{document}